\pdfoutput=1
\documentclass[final,leqno,onefignum,onetabnum,onealgnum,oneeqnum,onethmnum]{siamart171218}
\usepackage{lipsum}
\usepackage{amsfonts}
\usepackage{amsmath}
\usepackage{graphicx}
\usepackage{epstopdf}
\usepackage[caption=false]{subfig}
\graphicspath{{imgs/}}
\Crefname{Algorithm}{Algorithm}{Algorithms}
\crefformat{figure}{#2{}\scshape{Fig.} #1{}#3}
\crefformat{lemma}{#2{}\scshape{Lemma} #1{}#3}
\crefformat{theorem}{#2{}\scshape{Theorem} #1{}#3}
\crefformat{proposition}{#2{}\scshape{Proposition} #1{}#3}
\crefformat{definition}{#2{}\scshape{Definition} #1{}#3}
\newtheorem{remark}{Remark}
\crefformat{remark}{#2{}\scshape{Remark} #1{}#3}
\usepackage{algpseudocode}
\AppendGraphicsExtensions{.tif}
\usepackage{transparent}
\usepackage{tikz}
\usetikzlibrary{positioning,calc}
\usepackage{xcolor}
\definecolor{selfDefinedColor}{rgb}{0.8,0.1,0.4}
\usepackage{mwe}
\tikzset{blankboximg/.style={remember picture,white,ultra thick,draw,inner sep=0pt,outer sep=0pt}}
\tikzset{redboximg/.style={remember picture,red,ultra thick,draw,inner sep=0pt,outer sep=0pt}}
\tikzset{blueboximg/.style={remember picture,blue,ultra thick,draw,inner sep=0pt,outer sep=0pt}}
\tikzset{selfDefinedColorboximg/.style={remember picture,selfDefinedColor,ultra thick,draw,inner sep=0pt,outer sep=0pt,dashed}}
\tikzset{greenboximg/.style={remember picture,green,ultra thick,draw,inner sep=0pt,outer sep=0pt,dashed}}
\usepackage{etoolbox}
\apptocmd{\sloppy}{\hbadness 10000\relax}{}{}
\title{Iterative regularization algorithms for image denoising with the TV-Stokes model}
\author{
    Bin Wu\thanks{Department of Computer Science, Electrical Engineering and Mathematical Sciences, Western Norway University of Applied Sciences, Inndalsveien 28, 5063 Bergen, Norway (\email{bin.wu@hvl.no}, \email{talal.rahman@hvl.no}).}
    \and
    Leszek Marcinkowski\thanks{Faculty of Mathematics, University of Warsaw, Banacha 2, 02-097 Warszawa, Poland (\url{leszek.marcinkowski@mimuw.edu.pl}).}
    \and
    Xue-Cheng Tai\thanks{Department of Mathematics, University of Bergen, All\'{e}gaten 41, 5007 Bergen, Norway (\email{xue-cheng.tai@uib.no}).}
    \and
    Talal Rahman\footnotemark[1]
}
\headers{Iterative regularization with TV-Stokes}{B. Wu, L. Marcinkowski, X. C. Tai, and T. Rahman}

\begin{document}
\maketitle

\begin{abstract}
    We propose a set of iterative regularization algorithms for the TV-Stokes model to restore images from noisy images with Gaussian noise. These are some extensions of the iterative regularization algorithm proposed for the classical Rudin-Osher-Fatemi (ROF) model for image reconstruction, a single step model involving a scalar field smoothing, to the TV-Stokes model for image reconstruction, a two steps model involving a vector field smoothing in the first and a scalar field smoothing in the second. The iterative regularization algorithms proposed here are Richardson's iteration like. We have experimental results that show improvement over the original method in the quality of the restored image. Convergence analysis and numerical experiments are presented. 
\end{abstract}

\begin{keywords}
  Image processing, Total variation minimization
\end{keywords}

\begin{AMS}
  68Q25, 68R10, 68U05
\end{AMS}

\section{Introduction}
\label{intro}

Recovering an image from a noisy and blurry image is an inverse problem which is possible to be solved via variational methods, using total variation regularization, e.g., cf. \cite{Rudin1992,Chan1999,Chan2005,Aubert2006,Rahman2007,Zhu2008,Tai2009,Scherzer2009,Elo2009,Elo2009a,Yang2009a,Wu2010,Wu2011,He2010,Bredies2010,Litvinov2011,Chan2011a,Chen2012,Wu2012,Bayram2012,Hahn2012,Marcinkowski2016}. In this paper, we only focus on the denoising problem in image processing. Considering a noisy image $f:\Omega \mapsto \mathbb{R}$, where $\Omega$ is a bounded open subset of $\mathbb{R}^2$, the problem is to find a decomposition such that $f = u + v$, where $u$ is the signal, and $v$ is the noise. Let us consider this problem as an optimization problem. The simplest model can be the least square fitting, in other words, to find the minimizer in the squared $L^2$ space:
\begin{equation*}
    u = \mbox{arg} \min_{u} \Vert u - f \Vert_{(\Omega)}^2.
\end{equation*}
The notation $\Vert \cdot \Vert$, in this paper, refers to a $L^2$ norm if there is no specific subscript. This model, however, only works when we know the structure of $u$ otherwise there is only a trivial solution $u = f$. It is obvious that, without sufficient priori, to find a decomposition is an ill-posed inverse problem, cf. e.g., \cite{Chan2005,Aubert2006}. A regularizer is thus necessary. The Tikhonov regularizer is the first one used in this problem in history. In general, we define regularizer as 
\begin{equation*}
    J_p(u) := \int_{\Omega} \vert \nabla u \vert^p.
\end{equation*}
The Tikhonov regularizer is the case where $p = 2$. For the models with this regularizer, it is difficult to preserve edges while smoothing noise. Rudin, Osher, and Fatemi proposed a model (ROF) with a regularizer where $p = 1$. To be simplified in presentation, in this paper, we denote $J(\cdot)$ equipped with $p=1$ as default. The regularizer is thus the $BV$ seminorm where $BV(\Omega)$ means the space of functions with the bounded variation on $\Omega$. The ROF model successfully enhances the capability in edge preserving. However, it suffers a staircase effect which makes the restored image patternized. There are many models to overcome this problem, for instance, the high order regularization, LOT model, and TV-Stokes. The TV-Stokes model is defined as follows:
\begin{equation}\label{eq:tvs1}
    \boldsymbol{\tau} = 
    \mbox{arg} 
    \min_{\substack{\boldsymbol{\tau} \in BV(\Omega) \\ \nabla \cdot \boldsymbol{\tau} = 0}} 
    \bigg\{
        J(\boldsymbol{\tau}) + 
        H(\boldsymbol{\tau},\boldsymbol{\tau}^0)
    \bigg\},
\end{equation}
and
\begin{equation}\label{eq:tvs2}
    u = 
    \mbox{arg} 
    \min_{u \in BV(\Omega)} 
    \bigg\{ 
        J(u) -  
        \langle \nabla u,\frac{\boldsymbol{\tau}^{\bot}}{\vert\boldsymbol{\tau}^{\bot}\vert} \rangle +
        H(u,f)
   \bigg\},
\end{equation}
where $\boldsymbol{\tau},\boldsymbol{\tau}^0 \in \mathbb{R}^2$ are vectors, $\nabla \boldsymbol{\tau}$ inside $J(\boldsymbol{\tau})$ is a $2\times2$ matrix, i.e., the gradient of the vector $\boldsymbol{\tau}$; $H(u,f) := \frac{\eta}{2} \Vert u - f \Vert^2$ stands the quadratic fidelity with a scale parameter $\eta$; $\langle \cdot, \cdot \rangle$ denotes the inner product. In this model, a smoothed tangent field $\boldsymbol{\tau}$ is firstly obtained by solving the minimization problem (\ref{eq:tvs1}) under a divergence-free constraint, and subsequently the restored $u$ is obtained by a kind of vector matching under a limited deviation which is formulated as (\ref{eq:tvs2}). 

The ROF model is also considered defective in some cases for signal and noise decomposition, cf. e.g., \cite{Osher2005, Meyer2001}. There are many ways to handle this problem, for instance, Meyer's model, cf. \cite{Meyer2001}, the Vese and Osher's approximated Meyer's model, cf. \cite{Vese2003}, the Osher, Sol\'e, and Vese's model, cf. \cite{Osher2003}. There is another stream of methods which handle the problem through the iterative way, e.g., \cite{Osher2005} applies an iterative algorithm on the ROF model. To heuristically introduce this algorithm, we start with a typical ROF model as follows,
\begin{equation}\label{eq:rofitr1}
    u^1 = 
    \mbox{arg} 
    \min_{u \in BV(\Omega)} 	
    \bigg\{
        J(u) + 
        H(u,f)
    \bigg\}.
\end{equation}
We then calculate unit normal vector $\frac{\mathbf{n}^1}{\vert \mathbf{n}^1 \vert} =
\frac{\nabla u^1}{\vert \nabla u^1 \vert}$ and perform a vector matching step
\begin{equation}\label{eq:rofitr2}
    u^2 = 
    \mbox{arg} 
    \min_{u \in BV(\Omega)} 
    \bigg\{ 
        J(u) -  
        \langle \nabla u, \frac{\mathbf{n}^1}{\vert \mathbf{n}^1 \vert} \rangle +
        H(u,f)
    \bigg\},
\end{equation}
The optimal condition for (\ref{eq:rofitr1}), namely the Euler-Lagrange equation, is
\begin{equation}\label{eq:rofitr1-opt}
    -\nabla \cdot \frac{\nabla u^1}{\vert \nabla u^1 \vert} + \eta (u^1 - f) = 0.
\end{equation}
Slightly reforming (\ref{eq:rofitr2}) with adjoint, cf. \cite{Osher2005}, and substituting the relation (\ref{eq:rofitr1-opt}) into the reformed \cref{eq:rofitr2}, we get 
\begin{eqnarray*}
    u^2 &=&
    \mbox{arg} 
    \min_{u \in BV(\Omega)} 
    \bigg\{ 
        J(u) -  
        \langle u, \eta (f - u^1) \rangle +
        H(u,f)
    \bigg\}. \nonumber    
\end{eqnarray*}
By completing the square with some added constants, the above minimization problem is equivalent to the following
\begin{eqnarray}\label{eq:rofitr2_sub}
    u^2 &=& 
    \mbox{arg} 
    \min_{u \in BV(\Omega)} 
    \bigg\{ 
        J(u) +
        H(u,f + f - u^1)
    \bigg\}.    
\end{eqnarray}
It implies that the matching step exactly equivalents to a ROF model. The initial image $f$ is accordingly replaced by $f$ added with a `noise' $(f-u^1)$ obtained from the previous step. By an induction, \Cref{alg:osher} has been proposed by Osher and his coworkers, cf. \cite{Osher2005} for more details,.
\begin{algorithm}
    \caption{The iterative regularization proposed by Osher et al.}
    \label{alg:osher}
    \begin{algorithmic}[1]
        \State{Initialize $k = 0$, $v^{0} = 0$;}
        \Repeat
            \State{$u = \mbox{arg} \min_{u \in BV(\Omega)} \{ J(u) + H(u,f + v^k) \}$;}
            \State{Update noise: $v^{k+1} = f + v^k - u$;} 
            \State{$k = k + 1$;}
         \Until{satisfied;}
        \State \textbf{return} $u$. 
    \end{algorithmic} 
\end{algorithm}

Instead of considering the given image with accumulative noise, we consider a direct process on noise part in this paper. The idea can source back to the `twicing' method proposed by Tukey, which corrects the approximate solution obtained from the first step by repeating the same processing on its residual. We noted that this idea is the modified Richardson iteration which has been generalized to image restoration problems by Michael Charest Jr. and his coworkers based on scalar valued functional. By defining the operation of finding the solution of a minimization problem as $T(\cdot)$, and starting with initial iterate $u^0$, the analog of the modified Richardson iteration reads as follows,
\begin{equation}
    u^{k+1} = u^k + T(f-u^k),
\end{equation}
which is equivalent to the following in terms of residuals
\begin{equation}
    r_{ex}^{k+1} = r_{ex}^k - T(r_{ex}^k),
\end{equation}
where the exact residual is defined as $r_{ex}^k =  f - u^k$. It also can be derived that $u^k =
u^0 + \sum_{i=0}^{k-1} T(r_{ex}^i)$ and $r_{ex}^k = r_{ex}^0 - \sum_{0}^{k-1} T(r_{ex}^i)$. Let us call the above the Richardson-like iteration. In this paper, we present several Richardson-like iterative algorithms based on the TV-Stokes model.

The paper is organized as follows. Our main results including the proposed algorithms are in \cref{sec:algs}, experimental results are in \cref{sec:numt}, and the conclusions follow in \cref{sec:conclusions}.

\section{Proposed algorithms and their convergence analysis}
\label{sec:algs}

In this section, we present several variants of the iterative regularization algorithm for the TV-Stokes model. Those algorithms are quite simple and of the form of Richardson iteration. We prove their convergence based on the Bregman distance.

\subsection{Prelimits}

Before we start to present the algorithm, let first consider two equivalent minimization problems.

\begin{lemma}
    $\forall \boldsymbol{\tau} \in \mathbb{R}^2$, define operator $\Pi,$ such that
    $\Pi(\boldsymbol{\tau}) = (I - \nabla \triangle^{\dagger} \nabla \cdot)\boldsymbol{\tau}.$
    The constrained problem \cref{eq:tvs1} is equivalent to the following unconstrained problem
    \begin{equation}\label{eq:tvs1_1_equivalent}
        \boldsymbol{\tau} = 
        \mbox{arg} 
        \min_{\boldsymbol{\tau} \in BV(\Omega)} 	
        \bigg\{
            J(\Pi\boldsymbol{\tau}) + 
            H(\boldsymbol{\tau},\boldsymbol{\tau}^0)
        \bigg\}.
    \end{equation}
\end{lemma}

\begin{proof}
    Let $\mathbf{p}$ be the dual variable such that $\mathbf{p} \in
    C_c^1(\Omega, \mathbb{R}^4)$ and $\vert\mathbf{p}\vert \leq 1$, the minimization problem
    in \cref{eq:tvs1} is thus
    \begin{equation}\label{eq:tvs1_1_reform}
        \min_{\boldsymbol{\tau}} 
        \max_{\substack{\lambda, \mathbf{p} \\ \vert p \vert \leq 1}} 
        \left\{
            \int_{\Omega}
                \langle \boldsymbol{\tau},\nabla \cdot \mathbf{p}\rangle + 
                \frac{\eta}{2} (\boldsymbol{\tau} - \boldsymbol{\tau}^0)^2 + 
                \langle \lambda, \nabla \cdot \boldsymbol{\tau} \rangle
            d\mathbf{x}
        \right\},
    \end{equation}
    where $\lambda \in \mathbb{R}$ is the Lagrange multiplier. By the Minimax theorem, cf. \cite{Sion1957}, we can firstly
    consider the minimization problem with respect to $\boldsymbol{\tau}$ as well as the the maximization with respect to $\lambda$ freezing $\mathbf{p}$. The corresponding Euler-Lagrange equations are 
    \begin{equation}\label{eq:el_tvs1_1}
        \nabla \cdot \mathbf{p} + 
        \eta (\boldsymbol{\tau} - \boldsymbol{\tau}^0) - 
        \nabla \lambda 
        = 0,
    \end{equation}
    and
    \begin{equation*}
        \nabla \cdot \boldsymbol{\tau} = 0.
    \end{equation*}
    Taking the divergence for the both sides of \cref{eq:el_tvs1_1}, we obtain the following relation with the help of Moore-Penrose pseudoinverse
    \begin{equation*}
        \lambda = \Delta^{\dagger} \nabla \cdot \nabla \cdot \mathbf{p}.
    \end{equation*}
    The considered problem \cref{eq:tvs1_1_reform} is accordingly
    \begin{equation*}
        \min_{\boldsymbol{\tau}} 
        \max_{\substack{\mathbf{p} \\ \vert p \vert \leq 1}} 
        \left\{
            \int_{\Omega}
                \langle \boldsymbol{\tau},\nabla \cdot \mathbf{p} \rangle + 
                \frac{\eta}{2} (\boldsymbol{\tau} - \boldsymbol{\tau}^0)^2 + 
                \langle \Delta^{\dagger} \nabla \cdot \nabla \cdot \mathbf{p}, \nabla
                \cdot \boldsymbol{\tau} \rangle
            d\mathbf{x}
        \right\},
    \end{equation*}
    which is exactly same as follows with adjoint
    \begin{equation*}
        \min_{\boldsymbol{\tau}} 
        \max_{\substack{\mathbf{p} \\ \vert p \vert \leq 1}} 
        \left\{
            \int_{\Omega}
                \langle \boldsymbol{\tau},(I-\nabla \Delta^{\dagger} \nabla \cdot) \nabla
                \cdot \mathbf{p} \rangle + 
                \frac{\eta}{2} (\boldsymbol{\tau} - \boldsymbol{\tau}^0)^2 
            d\mathbf{x}
        \right\}.
    \end{equation*}
    Rewriting with $\Pi$ operator as we defined, it is
    \begin{equation*}
        \min_{\boldsymbol{\tau}} 
        \max_{\substack{\mathbf{p} \\ \vert p \vert \leq 1}} 
        \left\{
            \int_{\Omega}
                \langle \boldsymbol{\tau},\Pi(\nabla \cdot \mathbf{p}) \rangle + 
                \frac{\eta}{2} (\boldsymbol{\tau} - \boldsymbol{\tau}^0)^2 
            d\mathbf{x}
        \right\},
    \end{equation*}
    which is equivalent to the following primal problem
    \begin{equation*}
        \min_{\boldsymbol{\tau}} 
        \left\{
            \int_{\Omega}
                \vert \nabla \Pi(\boldsymbol{\tau}) \vert + 
                \frac{\eta}{2} (\boldsymbol{\tau} - \boldsymbol{\tau}^0)^2 
            d\mathbf{x}
        \right\}.
    \end{equation*}
\end{proof}

It is worth to mention that $\Pi(\cdot)$ is exactly the orthogonal projection to the divergence free subspace, cf. \cite{Elo2009}. The consequent property is that, for all $\boldsymbol{\tau} \in Y$ such that $Y = \{\mathbf{m}: \nabla \cdot \mathbf{m} = 0 \}$, we have $\Pi(\boldsymbol{\tau}) = \boldsymbol{\tau}$.

Let us recall the second step of the TV-Stokes model, cf. \cref{eq:tvs2}, as follows.
\begin{equation}\label{eq:tvs2m}
    u = 
    \mbox{arg} 
    \min_{u \in BV(\Omega)} 
    \bigg\{ 
        J(u) -  
        \alpha
        \langle 
            \nabla u,
            \frac{\boldsymbol{\tau}^{\bot}}{\vert\boldsymbol{\tau}^{\bot}\vert} 
        \rangle +
        H(u,f)
   \bigg\},
\end{equation}
where $\alpha$ is the parameter for orientation matching term $- \langle \nabla u, \frac{\boldsymbol{\tau}^{\bot}}{\vert\boldsymbol{\tau}^{\bot}\vert} \rangle$. When $\alpha=1$, the above minimization problem degenerates to \cref{eq:tvs2}. By completing the square, we can reform the above problem as follows
\begin{equation}\label{eq:tvs2mq}
    u = 
    \mbox{arg} 
    \min_{u \in BV(\Omega)} 
    \bigg\{ 
        J(u) +
        H(u,f - \cfrac{\alpha}{\eta} \nabla \cdot \frac{\boldsymbol{\tau}^{\bot}}{\vert\boldsymbol{\tau}^{\bot}\vert})
   \bigg\}.
\end{equation}
Observing (\ref{eq:tvs2mq}), we can find out that there is an optimal decomposition $f = u + \frac{\alpha}{\eta} \nabla \cdot \frac{\boldsymbol{\tau}^{\bot}}{\vert\boldsymbol{\tau}^{\bot}\vert}$ corresponding to the fidelity parameter $\eta / 2$. According to Meyer's theory, cf. \cite{Meyer2001}, $\frac{\alpha}{\eta} \nabla \cdot \frac{\boldsymbol{\tau}^{\bot}}{\vert\boldsymbol{\tau}^{\bot}\vert}$ can be read as the high frequency part, which represents the fine structures and the noise part of the corrupted image $f$. Since it is the indistinguishable part from noise for ROF model, we can roughly say it is also of Gaussian distribution with mean $0$ if the noise is white Gaussian, that is $\frac{\alpha}{\eta} \nabla \cdot \frac{\boldsymbol{\tau}^{\bot}}{\vert\boldsymbol{\tau}^{\bot}\vert} \sim N(0,\sigma^2)$ for some unknown variance $\sigma^2$. Another thing we can find out from (\ref{eq:tvs2mq}) is that the model results smooth $u$ with the variance $\sigma^2 = 1/\eta$, cf. e.g., \cite{Chambolle2009, Chen2012}, considering Gaussian noise only. Up to here, we can find that the TV-Stokes model tends to find an optimal image $u$ which is close to ``clean'' image $f  - \frac{\alpha}{\eta} \nabla \cdot \frac{\boldsymbol{\tau}^{\bot}}{\vert\boldsymbol{\tau}^{\bot}\vert}$, where the noisy part $\frac{\alpha}{\eta} \nabla \cdot \frac{\boldsymbol{\tau}^{\bot}}{\vert\boldsymbol{\tau}^{\bot}\vert}$ has an amplitude equal to its variance level while $\alpha = 1$, that is $\sigma^2 = 1/\eta$.

In our proposed algorithms, we only consider Richardson-like iterations, applied to the residual. For Gaussian noise, the residual $r$ satisfies $r \sim N(0, \sigma^2)$ for some unknown variance $\sigma^2$. It is natural to assume that a fixed percentage of the residual is the uncertain part, cf. $\alpha \in [0,1]$ in (\ref{eq:tvs2mq}). When $\alpha = 0$, the model reduces to the ROF model.

The following lemmas are necessary for the convergence analysis.

\begin{lemma}
\label{lemma:gs}
    Given $a,b,c \in L^2(\Omega;\mathbb{R})$, where $b \sim N(0, \sigma^2)$, $a \sim N(b, \sigma_1^2)$ and $c \sim N(b, \sigma_2^2)$, such that $\sigma_1^2 \leq \sigma_2^2$, then $\Vert a \Vert \leq \Vert c \Vert$.
\end{lemma}

\begin{proof}
    Since $a \sim N(b, \sigma_1^2)$, $a-b$ is a Gaussian distribution such that $a-b \sim N(0, \sigma_1^2)$. 
    Since $b$ and $a-b$ are two independent Gaussian distributions, their sum $a = b + (a - b)$ is another Gaussian distribution such that $a \sim N(0, \sigma^2 + \sigma_1^2)$.
    
    Similarly, the sum $c = b + (c - b)$ is also a Gaussian distribution such that $c \sim N(0, \sigma^2 + \sigma_2^2)$.
    
    And now, since $\sigma_1^2 \leq \sigma_2^2$, it follows that $\Vert a \Vert \leq \Vert c \Vert$.
\end{proof}

\begin{lemma}
\label{lemma:rof_fidelity_2}
    Suppose $u \in BV(\Omega;\mathbb{R})$. Consider two minimization problems same as the second step of TV-Stokes, i.e., $\min_u \{ J(u) - \alpha \langle \nabla u, \mathbf{v}/\vert\mathbf{v}\vert \rangle + \eta (u - f)^2 / 2 \}$, corresponding to two different fidelity parameters, $\eta_1$ and $\eta_2$, such that $\eta_1 \leq \eta_2$. If $u_1 = \arg \min_u \{ J(u) - \alpha \langle \nabla u, \mathbf{v}/\vert\mathbf{v}\vert \rangle + \eta_1 (u - f)^2 / 2 \}$ and $u_2 = \arg \min_u \{ J(u) - \alpha \langle \nabla u, \mathbf{v}/\vert\mathbf{v}\vert \rangle + \eta_2 (u - f)^2 / 2 \}$, then $\Vert u_1-f \Vert^2 \ge \Vert u_2 - f \Vert^2$.
\end{lemma}

\begin{proof}
    Rewrite the minimization problem for $\eta_2$ as follows.
    \begin{align}
    \label{eq:rof_u2_re_2}
        \begin{split}
            & \min_u \left \{ J(u) - \alpha \langle \nabla u, \cfrac{\mathbf{v}}{\vert\mathbf{v}\vert} \rangle + \cfrac{\eta_1}{2} (u - f)^2 + \cfrac{\eta_2 - \eta_1}{2} (u - f)^2 \right \}.
        \end{split}
    \end{align}
    Since $u_2$ is the minimizer of (\ref{eq:rof_u2_re_2}), the functional has following inequality
    \begin{align}
    \label{eq:rof_u2_re_ineq_2}
            & \int_{\Omega} J(u_2) - \alpha \langle \nabla u_2, \cfrac{\mathbf{v}}{\vert\mathbf{v}\vert} \rangle + \cfrac{\eta_1}{2} (u_2 - f)^2 + \cfrac{\eta_2 - \eta_1}{2} (u_2 - f)^2 \nonumber
                \\ 
            \leq & \int_{\Omega} J(u_1) - \alpha \langle \nabla u_1, \cfrac{\mathbf{v}}{\vert\mathbf{v}\vert} \rangle + \cfrac{\eta_1}{2} (u_1 - f)^2 + \cfrac{\eta_2 - \eta_1}{2} (u_1 - f)^2
                \\
            \leq & \int_{\Omega} J(u_2) - \alpha \langle \nabla u_2, \cfrac{\mathbf{v}}{\vert\mathbf{v}\vert} \rangle + \cfrac{\eta_1}{2} (u_2 - f)^2 + \cfrac{\eta_2 - \eta_1}{2} (u_1 - f)^2. \nonumber
    \end{align}
    Note that the first two terms in the functional are actually same as the minimization problem for $\eta_1$. $u_1$ therefore minimizes the energy composed by this two terms. 
    
    Comparing the first line and the last line of (\ref{eq:rof_u2_re_ineq_2}), we obtain $\Vert u_1-f \Vert^2 \ge \Vert u_2 - f \Vert^2$ by using relation $\eta_1 \leq \eta_2$.
\end{proof}

\begin{lemma}
\label{lemma:fidelity_constraint}
    Consider a minimization problems with ROF model, that is $r^* = \arg \min_r \{ J(r) + \eta (r - r^0)^2 / 2 \}$. Define $\eta = \beta/\gamma$ and $\beta \in (1,+\infty)$. There exists a constant $\gamma>0$ for any $r^0 \neq 0$ such that $r^* \neq 0$.
\end{lemma}

We will use the consequences from Meyer's theory, cf. \cite{Meyer2001}, to prove this lemma.
\begin{proof}
    According to Meyer's theory, \cite[Lemma~4 and Theorem~3 on p.~32]{Meyer2001}, if $\Vert r^0 \Vert_* > 1/\eta$, the ROF model generates a non-trivial decomposition, $r^0 = r^* + v$, that is $r^* \neq 0$ for any $r^0 \neq 0$.
    
    According to \cite[Lemma~3 on p.~31]{Meyer2001}, if $r^0 \in L^2(\mathbb{R}^2)$, then $\vert\int \tilde{r}(x) r^0(x) dx\vert \le \Vert\tilde{r}\Vert_{BV}\Vert r^0 \Vert_*$. By simply replacing $\tilde{r}(x)$ with $r^0(x)$, we obtain $\Vert r^0 \Vert^2 / \Vert r^0 \Vert_{BV} \le \Vert r^0 \Vert_*$. While $r^0 \neq 0$ and not constant over the entire domain $\Omega$, which is nature by considered problem, $\Vert r^0 \Vert^2 / \Vert r^0 \Vert_{BV} > 0$. Define $\gamma := \Vert r^0 \Vert^2 / \Vert r^0 \Vert_{BV}$. Since $\beta \in (1,+\infty)$ and $\eta = \beta/\gamma$, we have $\gamma > 1/\eta$. To sum up, the inequality, $\Vert r^0 \Vert_* > 1/\eta$, holds and, as the consequence from Meyer's theory, $r^* \neq 0$ while $r^0 \neq 0$. 
\end{proof}

\begin{remark}
\label{remark:fidelity_constraint}
{\sloppy For a discrete system, equipped with the finite center difference scheme for example, the divergence $\nabla \cdot \mathbf{g}$ can be expressed as follows.
    \begin{align*}
        \begin{split}
            r^0 = \nabla \cdot \mathbf{g} & = \partial_1 g_1 + \partial_2 g_2
            = \cfrac{g_1^+ - g_1^-}{2h} + \cfrac{g_2^+ - g_2^-}{2h},
        \end{split}
    \end{align*}
    where $^+$ and $^-$ denote the forward and backward positions, respectively. $h$ is the uniform discretized unit.
    We thus obtain the following inequality from the above definition.
    \begin{align*}
            \vert r^0 \vert^2 & = \left(\cfrac{g_1^+ - g_1^-}{2h} + \cfrac{g_2^+ - g_2^-}{2h}\right)^2
            \\
                  & \le 2 \left(\cfrac{g_1^+ + g_2^+}{2h}\right)^2 + 2\left(\cfrac{g_1^- + g_2^-}{2h}\right)^2
            \\
                  & \le \cfrac{1}{h^2} ( (g_1^+)^2 + (g_2^+)^2 + (g_1^-)^2 + (g_2^-)^2 )
            \\    & \le \cfrac{4}{h^2} ( \overline{g_1^2} + \overline{g_2^2} )
            \\    & \sim \cfrac{4}{h^2} ( g_1^2 + g_2^2 ),
    \end{align*}
    where $\overline{(\cdot)}$ denotes the average value over the finite volume.
    Consider the $L^{\infty}$ norm, we obtain
    \begin{align*}
        \begin{split}
            \Vert r^0 \Vert_{\infty} & \le \cfrac{2}{h} \Vert\mathbf{g}\Vert_{\infty}.
        \end{split}
    \end{align*}
    Since $\Vert r^0 \Vert_*$ is the infimum of $\Vert\mathbf{g}\Vert_{\infty}$, we find $\Vert r^0 \Vert_{\infty} h / 2 \le \Vert r^0 \Vert_*$ corresponding to $\gamma = \Vert r^0 \Vert_{\infty} h/2$. In practice, $h = 1$.}
\end{remark}

\subsection{Iterative regularization for the first step of TV-Stokes}

In this subsection, we only consider the Richardson-like iteration on the first step of TV-Stokes model. Following from the unconstrained problem \cref{eq:tvs1_1_equivalent}, with the help of $\Pi$ operator, our proposed Richardson-like algorithm is as follows, 
\begin{algorithm}
    \caption{Iterative regularization applied to the $1^{st}$ step of TV-Stokes}
    \label{alg:tvs1}
    \begin{algorithmic}[1]
        \State{Initialize $k = 0$, $\boldsymbol{\tau}= 0$, $\mathbf{r}_{ex}^0 = \nabla^{\bot} f$;}
        \Repeat
            \State{$k = k + 1$;}
            \State\label{alg:tvs1_itr}{$\mathbf{r}^{k} = \mbox{arg} \min_{\mathbf{r} \in BV(\Omega)} \{ 
            J(\Pi\mathbf{r}) + H(\mathbf{r},\mathbf{r}_{ex}^{k-1}) \}$;}
            \State\label{alg:tvs1_itr_update}{$\mathbf{r}_{ex}^{k}  = \mathbf{r}_{ex}^{k-1} -
            \mathbf{r}^{k}$;} 
            \State{$\boldsymbol{\tau} = \boldsymbol{\tau} + \mathbf{r}^{k}$;}
        \Until {satisfied;}
        \State{$u = \mbox{arg} \min_{u \in BV(\Omega)} \{ J(u) - \langle \nabla u,
        \frac{\boldsymbol{\tau}^{\bot}}{\vert\boldsymbol{\tau}^{\bot}\vert} \rangle + H(u,f) \}$;}
        \State \textbf{return} $u.$
    \end{algorithmic} 
\end{algorithm}

Let us define a convex functional $Q^{\mathbf{s}^{k-1}}(\mathbf{r})^k$ for each iteration in \Cref{alg:tvs1} with
\begin{equation}
    Q^{\mathbf{s}^{k-1}}(\mathbf{r})^k = 
    H(\mathbf{r},0) +
    J(\Pi\mathbf{r}) -
    J(\Pi\mathbf{r}^{k-1}) -
    \langle \mathbf{s}^{k-1}, \mathbf{r} - \mathbf{r}^{k-1} \rangle,
\end{equation}
where $\mathbf{r}^{k-1}$ denotes the minimizer for $Q^{\mathbf{s}^{k-2}}(\mathbf{r})^{k-1}$, and $\mathbf{r}^k_{ex} := \mathbf{r}^{k-1}_{ex} - \mathbf{r}^k$ is the exact residual, giving $\mathbf{r}^0_{ex} = \nabla^{\bot} f$. By defining $\mathbf{s}^{k-1} := \eta \mathbf{r}_{ex}^{k-1}$, considering the problem $\mathbf{r}^{k} = \mbox{arg} \min_{\mathbf{r} \in BV(\Omega)} Q^{\mathbf{s}^{k-1}}(\mathbf{r})^k$, we have
\begin{align}
    \mathbf{r}^{k} &= \mbox{arg} \min_{\mathbf{r} \in BV(\Omega)} Q^{\mathbf{s}^{k-1}}(\mathbf{r})^k \nonumber \\
                   &= \mbox{arg} \min_{\mathbf{r} \in BV(\Omega)} \{
                        H(\mathbf{r},0) +
                        J(\Pi\mathbf{r}) -
                        J(\Pi\mathbf{r}^{k-1}) -
                        \langle \eta \mathbf{r}_{ex}^{k-1}, \mathbf{r} - \mathbf{r}^{k-1} \rangle \}
                        \nonumber \\
                   &= \mbox{arg} \min_{\mathbf{r} \in BV(\Omega)} \{
                        H(\mathbf{r},0) +
                        J(\Pi\mathbf{r}) -
                        \langle \eta \mathbf{r}_{ex}^{k-1}, \mathbf{r} \rangle \}
                        \nonumber \\   
                   &=  \mbox{arg} \min_{\mathbf{r} \in BV(\Omega)} \{ 
                        H(\mathbf{r},\mathbf{r}_{ex}^{k-1}) +
                        J(\Pi\mathbf{r}) \}, \label{eq:tvs1-min-k}
\end{align}
which implies the considered problem $\mathbf{r}^{k} = \mbox{arg} \min_{\mathbf{r} \in BV(\Omega)} Q^{\mathbf{s}^{k-1}}(\mathbf{r})^k$ is equivalent to the problem listed on the line \ref{alg:tvs1_itr} in \Cref{alg:tvs1} since the terms in $Q^{\mathbf{s}^{k-1}}(\mathbf{r})^k$ without $\mathbf{r}$ are constants for the $k^{th}$ iteration. The \Cref{alg:tvs1} is therefore can be reformed as \Cref{alg:tvs1.2}
\begin{algorithm}
    \caption{Bregmanized version of the iterative regularization \Cref{alg:tvs1}}
    \label{alg:tvs1.2}
    \begin{algorithmic}[1]
        \State{Initialize $k = 0$, $\boldsymbol{\tau}=0$, $\mathbf{r}_{ex}^0 = \nabla^{\bot} f$;}
        \Repeat
            \State{$k = k + 1$;}
            \State\label{alg:tvs1_itr_reform}{$\mathbf{r}^{k} = \mbox{arg} \min_{\mathbf{r} \in BV(\Omega)}
            Q^{\mathbf{s}^{k-1}}(\mathbf{r})^k$;}
            \State\label{alg:tvs1_itr_update_reform}{$\mathbf{r}_{ex}^{k}  = \mathbf{r}_{ex}^{k-1} -
            \mathbf{r}^{k}$;} 
            \State{$\mathbf{s}^k = \eta \mathbf{r}_{ex}^k $;}
            \State{$\boldsymbol{\tau} = \boldsymbol{\tau} + \mathbf{r}^{k}$;}
        \Until{satisfied;}
        \State{$u = \mbox{arg} \min_{u \in BV(\Omega)} \{ J(u) - \langle \nabla u,
        \frac{\boldsymbol{\tau}^{\bot}}{\vert\boldsymbol{\tau}^{\bot}\vert} \rangle + H(u,f) \}$;}
        \State \textbf{return} $u.$
    \end{algorithmic} 
\end{algorithm}

\subsubsection{Well-definedness of iterates}

Let us start with a simple case without iteration, specifically for a fixed $k$. The considered minimization problem is \cref{eq:tvs1-min}. For such a given problem, we can find the solution exists and is unique.
\begin{lemma}\label{lemma:unique-solution}
    Let $\boldsymbol{\mathcal{R}} = \{ \mathbf{r} \vert \mathbf{r} \in BV(\Omega; \mathbb{R}^2), \Pi \mathbf{r} = \mathbf{r} \}$, $F(\mathbf{r}) = J(\Pi\mathbf{r}) + 
            H(\mathbf{r},\mathbf{z})$ and $\Pi \mathbf{z} = \mathbf{z}$.
    Consider the problem to find $\mathbf{r}^*$ such that 
        \begin{equation}\label{eq:tvs1-min}
        \mathbf{r}^* \in \boldsymbol{\mathcal{R}}, \quad F(\mathbf{r}^*) = \inf_{\mathbf{r} \in \boldsymbol{\mathcal{R}}} F(\mathbf{r}).
    \end{equation}
    The solution for this problem exists and is unique.
\end{lemma}

\begin{proof}
    Let 
    \begin{equation}\label{eq:inf1}
        m := \inf_{\mathbf{r} \in \boldsymbol{\mathcal{R}}} F(\mathbf{r}),
    \end{equation}
    and  $\{\mathbf{r}^j\}$ is a minimizing sequence such that
    \begin{equation}\label{eq:seq1}
        \mathbf{r}^j \in \boldsymbol{\mathcal{R}}, \quad \lim_{j \rightarrow + \infty} F(\mathbf{r}^j) \rightarrow m.
    \end{equation} Define an equivalent $BV$-norm \cite{Eggermont2009} as 
    \begin{equation}
        \Vert\mathbf{r}\Vert_{BV(\Omega;\mathbb{R}^2)} = \int_{\Omega} \vert\nabla \mathbf{r}\vert d\mathbf{x} + \Vert \mathbf{r} \Vert.
    \end{equation}
    Followed from the above definition, we have that the sequence $\{\mathbf{r}^j\}$ is bounded in $BV(\Omega; \mathbb{R}^2)$, and consequently there exists a convergent subsequence $\{\mathbf{r}^k\}$ such that
    \begin{equation}
        \mathbf{r}^k \underset{\boldsymbol{\mathcal{R}}}{\rightharpoonup} \hat{\mathbf{r}}
    \end{equation}
    By the lower semicontinuity of $F$, we have
    \begin{equation}
        \underset{\mathbf{r}^k \rightharpoonup \hat{\mathbf{r}}}{\underline{\lim}} F(\mathbf{r}^k) \ge F(\hat{\mathbf{r}})
    \end{equation}
    Due to \cref{eq:seq1}, we obtain 
    \begin{equation}
        m \ge F(\hat{\mathbf{r}}).
    \end{equation}    
    But owing to \cref{eq:inf1}, $m \le F(\hat{\mathbf{r}})$. Consequently $\hat{\mathbf{r}}$ is indeed a minimizer.
    Furthermore, because $F$ is strictly convex, the solution is unique.
\end{proof}

\begin{proposition}
    Setting $\mathbf{r}^0_{ex} = \nabla^{\bot}f$, $\mathbf{s}^0 := \eta \mathbf{r}^0_{ex}$, and 
    $\mathbf{q}^k = \partial H(\mathbf{r}^k,0) = \eta \mathbf{r}^k$, for each $k \in \mathbb{N}$, there is
    an unique minimizer $\mathbf{r}^k$ of $Q^{\mathbf{s}^{k-1}}(\mathbf{r})^k$, and a subgradient $\mathbf{s}^k \in \partial 
    J(\Pi \mathbf{r}^k)$ such that 
    \begin{equation}\label{eq:s_relation1}
        \mathbf{s}^k + \mathbf{q}^k = \mathbf{s}^{k-1}
    \end{equation}
\end{proposition}

\begin{proof}
    The well-definedness for each iteration follows directly from \cref{lemma:unique-solution}. The relation between $\mathbf{s}$ and $\mathbf{q}$ is proved by induction. For k = 1, we have 
    \begin{equation*}
        \mbox{arg} \min_{\mathbf{r} \in BV(\Omega)} Q^{\mathbf{s}^0}(\mathbf{r})^1 = \mbox{arg} \min_{\mathbf{r} \in
        BV(\Omega)} \{ H(\mathbf{r},\mathbf{r}_{ex}^0) + J(\Pi\mathbf{r}) \}.
    \end{equation*} 
    The relation $\mathbf{s}^1 + \mathbf{q}^1 = \mathbf{s}^0$ holds by defining $\mathbf{s}^1 = \eta \mathbf{r}_{ex}^1$ which exactly can be deduced by the relations of $\mathbf{q}^1 = \eta \mathbf{r}^1$, $\mathbf{s}^0 = \eta \mathbf{r}_{ex}^0$ and $\mathbf{r}_{ex}^1 = \mathbf{r}_{ex}^0 - \mathbf{r}^1$. Taking the observation of optimal condition for the case $k = 1$
    \begin{equation*}
        \partial J(\Pi \mathbf{r}^1) + \eta \mathbf{r}^1 - \mathbf{s}^0 \ni 0,
    \end{equation*}
    which is the same as 
    \begin{equation*}
        \partial J(\Pi \mathbf{r}^1) \ni \eta \mathbf{r}_{ex}^1,
    \end{equation*}    
    we have $\mathbf{s}^1 \in \partial J(\Pi \mathbf{r}^1)$, cf. \cite{Osher2005} where $\partial J$ is the subgradient of $J$ in Euclidean space $\mathbb{R}^2$. Assuming that $\mathbf{s}^{k-1} = \eta \mathbf{r}_{ex}^{k-1} \in \partial J(\Pi \mathbf{r}^{k-1})$ holds, the $k^{th}$ case is 
    \begin{eqnarray*}
        &  \mbox{arg} \min_{\mathbf{r} \in BV(\Omega)} Q^{\mathbf{s}^{k-1}}(\mathbf{r})^k & \\
        &=& \mbox{arg} \min_{\mathbf{r} \in BV(\Omega)} 
            \{\frac{\eta}{2} \mathbf{r}^2 + 
                J(\Pi \mathbf{r}) - 
                J(\Pi \mathbf{r}^{k-1}) -
                \langle \eta \mathbf{r}_{ex}^{k-1}, \mathbf{r} - \mathbf{r}^{k-1} \rangle \} \\
        &=& \mbox{arg} \min_{\mathbf{r} \in BV(\Omega)} 
            \{\frac{\eta}{2}(\mathbf{r}-\mathbf{r}_{ex}^{k-1})^2 + 
            J(\Pi \mathbf{r}) \}.
    \end{eqnarray*}
    The optimal condition is accordingly
    \begin{equation*}
        \eta \mathbf{r}^k -
        \mathbf{s}^{k-1} +
        \partial J(\Pi\mathbf{r}^k) 
        \ni 0.
    \end{equation*}
    Since $\mathbf{s}^{k-1} = \eta \mathbf{r}_{ex}^{k-1}$ and $\mathbf{q}^k = \partial H(\mathbf{r}^k, 0) = \eta \mathbf{r}^k$, it is easy to find that $\mathbf{s}^k = \eta (\mathbf{r}_{ex}^{k-1} - \mathbf{r}^k) = \eta \mathbf{r}_{ex}^k \in \partial J(\Pi\mathbf{r}^k)$, and thus we obtain
    \cref{eq:s_relation1}.
\end{proof}

\subsubsection{Convergence analysis}

We define the generalized Bregman distance associated with $J(\Pi(\cdot))$ as follows 
\begin{equation*}
    D^{\mathbf{s}}(\mathbf{w},\mathbf{m}) := 
    J(\Pi\mathbf{w}) - 
    J(\Pi\mathbf{m}) -
    \langle \mathbf{s}, \mathbf{w} - \mathbf{m} \rangle,
\end{equation*}
where $\mathbf{s}$ is the subgradient for $J(\Pi\mathbf{m})$.
\begin{proposition}
    The sequence $H(\mathbf{r}^k,0)$ is monotonically nonincreasing, and 
    \begin{subequations}
        \begin{align}
            &H(\mathbf{r}^k,0) 
            \leq H(\mathbf{r}^k,0) +
                D^{\mathbf{s}^{k-1}}(\mathbf{r}^k,\mathbf{r}^{k-1}) 
            \leq H(\mathbf{r}^{k-1},0),\label{eq:d_relation1} \\
            &D^{\mathbf{s}^k}(\mathbf{r},\mathbf{r}^k) + 
            D^{\mathbf{s}^{k-1}}(\mathbf{r}^k,\mathbf{r}^{k-1}) + 
            H(\mathbf{r}^k,0) 
            \leq H(\mathbf{r},0) + 
                D^{\mathbf{s}^{k-1}}(\mathbf{r},\mathbf{r}^{k-1}),\label{eq:d_relation2}
        \end{align}
    \end{subequations}
    subject to $k \in \mathbb{N} \setminus \{ 1 \}$.
\end{proposition}

\begin{proof}
    Since $D^{\mathbf{s}^{k-1}}(\mathbf{r}^k,\mathbf{r}^{k-1})$ is nonnegative, it is easy to find 
    \begin{equation*}
        H(\mathbf{r}^k,0) 
        \leq H(\mathbf{r}^k,0) +
            D^{\mathbf{s}^{k-1}}(\mathbf{r}^k,\mathbf{r}^{k-1}) 
        = Q^{\mathbf{s}^{k-1}}(\mathbf{r}^k)^k.
    \end{equation*}
    Because $\mathbf{r}^k$ is the minimizer of $Q^{\mathbf{s}^{k-1}}(\mathbf{r})^k$, we have 
    \begin{equation*}
        Q^{\mathbf{s}^{k-1}}(\mathbf{r}^k)^k 
        \leq Q^{\mathbf{s}^{k-1}}(\mathbf{r}^{k-1})^k 
        = H(\mathbf{r}^{k-1},0),
    \end{equation*}
    which implies \cref{eq:d_relation1}.
    \begin{eqnarray*}
            &&
                D^{\mathbf{s}^k}(\mathbf{r},\mathbf{r}^k) 
                - D^{\mathbf{s}^{k-1}}(\mathbf{r},\mathbf{r}^{k-1}) 
                + D^{\mathbf{s}^{k-1}}(\mathbf{r}^k,\mathbf{r}^{k-1}) \\
        =   && 
                J(\Pi\mathbf{r}) 
                - J(\Pi\mathbf{r}^k) 
                - \langle \mathbf{s}^k, \mathbf{r} - \mathbf{r}^k \rangle \\ 
            && 
                - J(\Pi\mathbf{r}) 
                + J(\Pi\mathbf{r}^{k-1}) 
                + \langle \mathbf{s}^{k-1},\mathbf{r}-\mathbf{r}^{k-1} \rangle \\ 
            && 
                + J(\Pi\mathbf{r}^k) 
                - J(\Pi\mathbf{r}^{k-1}) 
                - \langle \mathbf{s}^{k-1}, \mathbf{r}^k - \mathbf{r}^{k-1} \rangle \\
        =   && 
                \langle \mathbf{s}^{k-1} - \mathbf{s}^k,  \mathbf{r} - \mathbf{r}^k \rangle \\ 
        =   &&  \langle \mathbf{q}^k,  \mathbf{r} - \mathbf{r}^k \rangle.
    \end{eqnarray*}   
    The relation $\mathbf{s}^{k-1} - \mathbf{s}^k = \mathbf{q}^k$ has been used here according to Proposition \ref{eq:s_relation1}. The $\mathbf{q}^k$ is the subgradient of $H(\mathbf{r}^k,0)$. By the definition of subgradient, we have
    \begin{equation*}
        D^{\mathbf{s}^k}(\mathbf{r},\mathbf{r}^k) -
        D^{\mathbf{s}^{k-1}}(\mathbf{r},\mathbf{r}^{k-1}) + 
        D^{\mathbf{s}^{k-1}}(\mathbf{r}^k,\mathbf{r}^{k-1})
        = \langle \mathbf{q}^k,  \mathbf{r} - \mathbf{r}^k \rangle 
        \leq H(\mathbf{r},0) - 
            H(\mathbf{r}^k,0),
    \end{equation*}
    and thus we obtain \cref{eq:d_relation2}.
\end{proof}

There is a direct result from this relation \cref{eq:d_relation2}. If there exists a minimizer $\mathbf{r}$ of $H(\cdot,0)$, by using \cref{eq:d_relation2}, we have
\begin{eqnarray}\label{eq:d_relation3}
    D^{\mathbf{s}^k}(\mathbf{r},\mathbf{r}^k) 
    &\leq& D^{\mathbf{s}^k}(\mathbf{r},\mathbf{r}^k) + 
        D^{\mathbf{s}^{k-1}}(\mathbf{r}^k,\mathbf{r}^{k-1}) \nonumber \\ 
    &\leq& D^{\mathbf{s}^k}(\mathbf{r},\mathbf{r}^k) + 
        D^{\mathbf{s}^{k-1}}(\mathbf{r}^k,\mathbf{r}^{k-1}) +
        H(\mathbf{r}^k,0) -
        H(\mathbf{r},0) \nonumber \\
    &\leq& D^{\mathbf{s}^{k-1}}(\mathbf{r},\mathbf{r}^{k-1}).   
\end{eqnarray}
It implies that, for each iteration, the Bregman distance to optimal $\mathbf{r}$ is getting shorter.

\begin{theorem}
If $\mathbf{r} \in BV(\Omega; \mathbb{R}^2)$ is the minimizer of $H(\cdot,0)$ subject to $k \in \mathbb{N} \setminus \{ 1 \}$, then $\mathbf{r}^k$ converges and
\begin{equation}\label{eq:h_relation1}
    H(\mathbf{r}^k,0) \leq \frac{J(\Pi \mathbf{r}) - J(\Pi 
    \mathbf{r}^1) - \langle \mathbf{s}^1, \mathbf{r} - 
    \mathbf{r}^1 \rangle}{k-1},
\end{equation}
moreover, 
\begin{equation*}
    \boldsymbol{\tau}^k = \sum_{i=1}^{k} \mathbf{r}^{i},
\end{equation*} 
converges to $\boldsymbol{\tau}^0$.
\end{theorem}
\begin{proof}
Taking the sum of \cref{eq:d_relation2}, we obtain
\begin{equation}\label{eq:d_relation4}
        D^{\mathbf{s}^k}(\mathbf{r},\mathbf{r}^k) + 
        \sum_{i=2}^k 
        \left[
            D^{\mathbf{s}^{i-1}}(\mathbf{r}^i,\mathbf{r}^{i-1}) + 
            H(\mathbf{r}^i,0) -
            H(\mathbf{r},0)
        \right]
        \leq  D^{\mathbf{s}^1}(\mathbf{r},\mathbf{r}^1).  
\end{equation}
Since $H(\mathbf{r}^k,0)$ is monotonically nonincreasing,
\begin{equation*}
        (k - 1) 
        \left[
            D^{\mathbf{s}^{k-1}}(\mathbf{r}^k,\mathbf{r}^{k-1}) + 
            H(\mathbf{r}^k,0) -
            H(\mathbf{r},0)
        \right]
        \leq J(\Pi \mathbf{r}) - J(\Pi \mathbf{r}^1) - \langle 
        \mathbf{s}^1, \mathbf{r} - \mathbf{r}^1 \rangle.  
\end{equation*}
Because $\mathbf{r}$ is the minimizer of $H(\cdot,0)$ and $D^{\mathbf{s}^{k-1}}(\mathbf{r}^k,\mathbf{r}^{k-1})$ is nonnegative, we obtain \cref{eq:h_relation1}. It implies that, when $k \rightarrow \infty$, $\mathbf{r}^k$ converges to $0$ with rate
\begin{equation*}
    \Vert\mathbf{r}^k\Vert
    \leq \sqrt{ \frac{J(\Pi \mathbf{r}) - J(\Pi 
            \mathbf{r}^1) - \langle \mathbf{s}^1, \mathbf{r} - 
            \mathbf{r}^1 \rangle}{k-1} } 
    = \mathcal{O}((k-1)^{-1/2}). 
\end{equation*}
From the definition,
\begin{subequations}
    \begin{align}
    \mathbf{r}^k = \mbox{arg} 
        \min_{\mathbf{r}}
            \left\{
                \int_{\Omega}
                    \vert \nabla \Pi(\mathbf{r}) \vert + 
                    \frac{\eta}{2} (\mathbf{r} - (\nabla^{\bot}f - \sum_{i=1}^{k-1}
                        \mathbf{r}^{i}))^2 
                d\mathbf{x}
            \right\}, \nonumber
    \end{align}            
\end{subequations}
we obtain the optimal condition
\begin{equation*}
    \partial J(\Pi(\mathbf{r}^k)) +
    \eta (\mathbf{r}^k - (\nabla^{\bot}f - \sum_{i=1}^{k-1} \mathbf{r}^{i}))^2)) \ni 0.
\end{equation*}
Since $\mathbf{r}^k$ converges to $0$ while $k \rightarrow \infty$ and, and therefor $\partial J(\Pi(\mathbf{r}^{\infty})) = 0$, we obtain $\sum_{i=1}^{\infty}\mathbf{r}^{i}=\nabla^{\bot}f$.
\end{proof}

\subsection{Iterative regularization for the second step of TV-Stokes}

In this subsection, we consider the Richardson iteration on the second step of TV-Stokes model. Similar to \Cref{alg:tvs1}, the proposed algorithm is accordingly the following as \Cref{alg:tvs2}.

\begin{algorithm}
    \caption{Iterative regularization applied to the $2^{nd}$ step of TV-Stokes}
    \label{alg:tvs2}
    \begin{algorithmic}[1]
        \State{$\boldsymbol{\tau} = \mbox{arg} \min_{\boldsymbol{\tau} \in BV(\Omega; \mathbb{R}^2)} \{ J(\Pi \boldsymbol{\tau}) + 
            H(\boldsymbol{\tau},\nabla^{\bot} f) \}$;}
        \State{Initialize $k = 0$, $u = 0$, $r_{ex}^0 = f$;}
        \Repeat
            \State{$k = k + 1$;}
            \State\label{eq:tvs2_itr}{$r^k = \mbox{arg} \min_{r \in BV(\Omega; \mathbb{R})} \{ J(r) + 
            H^k(r,r_{ex}^{k-1} - \cfrac{\alpha}{\eta^k} \nabla \cdot \frac{\boldsymbol{\tau}^{\bot}}{\vert\boldsymbol{\tau}^{\bot}\vert}) \}$;}
            \State\label{eq:tvs2_itr_update}{$r_{ex}^k  = r_{ex}^{k-1} - r^k$;} 
            \State{$u = u + r^k$;}
        \Until{satisfied;}
        \State \textbf{return} $u$.
    \end{algorithmic} 
\end{algorithm}

\subsubsection{Well-definedness of iterates}
Let us start with a simple case without iteration, specifically for a fixed $k$. The considered minimization problem is shown as (\ref{eq:tvs2-min}). For a such given problem, we can find the solution exists
and is unique.
\begin{lemma}\label{lemma:unique-solution2}
    Let $\mathcal{R} = \{ r \vert r \in BV(\Omega; \mathbb{R}) \}$, $F(r) = J(r) + 
            H(r,z)$.
    Consider the problem to find $r^*$ such that 
        \begin{equation}\label{eq:tvs2-min}
        r^* \in \mathcal{R}, \quad F(r^*) = \inf_{r \in \mathcal{R}} F(r).
    \end{equation}
    The solution for this problem exists and is unique.
\end{lemma}

\begin{proof}
    Let 
    \begin{equation}\label{eq:inf2}
        \hat{m} := \inf_{r \in \mathcal{R}} F(r),
    \end{equation}
    and  $\{r^j\}$ is a minimizing sequence such that
    \begin{equation}\label{eq:seq2}
        r^j \in \mathcal{R}, \quad \lim_{j \rightarrow + \infty} F(r^j) \rightarrow \hat{m}.
    \end{equation} Define an equivalent $BV$-norm as 
    \begin{equation}
        \Vert r \Vert_{BV(\Omega;\mathbb{R})} = \int_{\Omega} \vert\nabla r\vert d\mathbf{x} + \Vert r \Vert.
    \end{equation}
    Followed from the above definition, we have that the sequence $\{r^j\}$ is bounded in $BV(\Omega; \mathbb{R})$, and consequently there exists a convergent sub-sequence $\{r^k\}$ such that
    \begin{equation}
        r^k \underset{\mathcal{R}}{\rightharpoonup} \hat{r}
    \end{equation}
    By the lower semicontinuity of $F$, we have
    \begin{equation}
        \underset{r^k \rightharpoonup \hat{r}}{\underline{\lim}} F(r^k) \ge F(\hat{r})
    \end{equation}
    Due to relation (\ref{eq:seq2}), we obtain 
    \begin{equation}
        \hat{m} \ge F(\hat{r}).
    \end{equation}    
    But owing to \cref{eq:inf2}, $\hat{m} \le F(\hat{r})$. Consequently $\hat{r}$ is indeed a minimizer.
    Furthermore, because $F$ is strictly convex, the solution is unique.
\end{proof}

The well-definedness for each iteration follows directly from the above \cref{lemma:unique-solution2}. In the iterations, we choose the fidelity parameter $\eta^k$ for each iteration as shown in \cref{lemma:fidelity_constraint} such that $\eta^k = \max ( \beta / \gamma, \eta^{k-1})$, where $\beta \in (1,+\infty)$.

Consider the following minimizations for iterations $k$ and $k+1$ ($k \in \mathbb{N}$)
    \begin{align}
        \begin{split}
            \label{eqs:tvs2m_k}
            & r^k = \mbox{arg} \min_{r \in BV(\Omega; \mathbb{R})} \left \{ J(r) + H^k(r,r_{ex}^{k-1} - \cfrac{\alpha}{\eta^k} \nabla \cdot \frac{\boldsymbol{\tau}^{\bot}}{\vert\boldsymbol{\tau}^{\bot}\vert}) \right \};
        \end{split}
        \\
        \begin{split}
            \label{eqs:tvs2m_tk1}
            & \tilde{r}^{k+1} = \mbox{arg} \min_{r \in BV(\Omega; \mathbb{R})} \left \{ J(r) + H^k(r,r_{ex}^k - \cfrac{\alpha}{\eta^k} \nabla \cdot \frac{\boldsymbol{\tau}^{\bot}}{\vert\boldsymbol{\tau}^{\bot}\vert}) \right \};
        \end{split}
        \\
        \begin{split}
            \label{eqs:tvs2m_k1}
            & r^{k+1} = \mbox{arg} \min_{r \in BV(\Omega; \mathbb{R})} \left \{ J(r) + H^{k+1}(r,r_{ex}^k - \cfrac{\alpha}{\eta^{k+1}} \nabla \cdot \frac{\boldsymbol{\tau}^{\bot}}{\vert\boldsymbol{\tau}^{\bot}\vert}) \right \}.
        \end{split}
    \end{align} 
The Euler-Lagrangian equation for $k$ iteration (\ref{eqs:tvs2m_k}) is 
\begin{align*}
    \begin{split}
        & \partial J(r^k) + \eta^k ( r^k - r_{ex}^{k-1} + \cfrac{\alpha}{\eta^k} \nabla \cdot \frac{\boldsymbol{\tau}^{\bot}}{\vert\boldsymbol{\tau}^{\bot}\vert} )  \ni 0.
    \end{split}
\end{align*}
The subgradient of $J$ thus can be determined as $s^k := \partial J(r^k)  = \eta^k (r_{ex}^k - \frac{\alpha}{\eta^k} \nabla \cdot \frac{ \boldsymbol{\tau} ^ {\bot} }{ \vert \boldsymbol{\tau}^{\bot} \vert } )$ for $k \in \mathbb{N}$. When $k = 1$, we set $r_{ex}^0 := f$.

\begin{definition}
\label{def:sub_g_itr}
    Define $\tilde{Q}^{\tilde{s}^k}(r)^{k+1} := H^k(r,0) + J(r) - J(r^k) - \langle s^k, r - r^k \rangle$. Setting $\tilde{q}^{k+1} := \partial H^k(\tilde{r}^{k+1},0) = \eta^k \tilde{r}^{k+1}$, for each $k \in \mathbb{N}$, there is an unique minimizer $\tilde{r}^{k+1}$ of $\tilde{Q}^{\tilde{s}^k}(r)^{k+1}$, and a subgradient $\tilde{s}^{k+1} \in \partial J(\tilde{r}^{k+1})$  such that 
    \begin{equation}\label{eq:s_relation2}
        \tilde{s}^{k+1} + \tilde{q}^{k+1} = s^k
    \end{equation}
\end{definition}

The relation of $\tilde{s}$, $s$ and $\tilde{q}$ is easy to be obtained.

\begin{lemma}
\label{lemma:exact_itr}
    For a given $k \in \mathbb{N}\setminus \{1\}$, assume $r^k_{ex},r^{k+1}\in L^2(\Omega;\mathbb{R})$ such that $r^{k+1} \sim N(0, (\sigma^{k+1})^2)$, $r^k_{ex} \sim N(r^{k+1}, (\sigma^k_{ex})^2)$, then $\Vert r^{k+1}_{ex} \Vert \leq \Vert r^k_{ex} \Vert$.
\end{lemma}

\begin{proof}
    Since $r^k_{ex} \sim N(r^{k+1}, (\sigma^k_{ex})^2)$, $r^{k+1}_{ex} = r^k_{ex} - r^{k+1}$ is a Gaussian distribution such that $r^{k+1}_{ex} \sim N(0, (\sigma^k_{ex})^2)$. $r^{k+1}$ as given is also a Gaussian distribution such that $r^{k+1} \sim N(0, (\sigma^{k+1})^2)$. Since $r^{k+1}$ and $r^k_{ex}-r^{k+1}$ are two independent Gaussian distributions, the sum $r^k_{ex} = r^{k+1} + (r^k_{ex} - r^{k+1})$ is another Gaussian distribution such that $r^k_{ex} \sim N(0, (\sigma^k_{ex})^2 + (\sigma^{k+1})^2)$.
    
    We obtain $\Vert r^{k+1}_{ex} \Vert \leq \Vert r^k_{ex} \Vert$ since $(\sigma^k_{ex})^2 \leq (\sigma^k_{ex})^2 + (\sigma^{k+1})^2$.
\end{proof}

\subsubsection{Convergence analysis}

We define the generalized Bregman distance associated with $J(\cdot)$ as follows 
\begin{align*}
    \begin{split}
        & D^{s}(w,m) := J(w) - J(m) - \langle s, w - m \rangle.
    \end{split}
\end{align*}
\begin{proposition}
    \begin{align}
        \begin{split}
        \label{eq:d_relation8}
            & H^k(\tilde{r}^{k+1},0) \leq H^k(\tilde{r}^{k+1},0) + D^{s^k}(\tilde{r}^{k+1},r^k) \leq H^k(r^k,0),
        \end{split}
        \\
        \begin{split}
        \label{eq:d_relation9}
            & D^{\tilde{s}^{k+1}}(r,\tilde{r}^{k+1}) + D^{\tilde{s}^k}(\tilde{r}^{k+1},\tilde{r}^k) + H^k(\tilde{r}^{k+1},0)
            \\
       \leq & \cfrac{\eta^k - \eta^{k-1}}{\eta^k} H^k(r^k_{ex},0) + D^{\tilde{s}^k}(r,\tilde{r}^k) + (\eta^k - \eta^{k-1}) \langle \tilde{r}^{k+1} - r^k_{ex}, r \rangle,
        \end{split}
    \end{align}
    subject to $J(r) < \infty$ and $k \in \mathbb{N}$.
\end{proposition}
\begin{proof}
    Since $D^{s^k}(\tilde{r}^{k+1},r^k)$ is non-negative, it is easy to find 
    \begin{align*}
        \begin{split}
            & H^k(\tilde{r}^{k+1},0) \leq H^k(\tilde{r}^{k+1},0) + D^{s^k}(\tilde{r}^{k+1},r^k) = \tilde{Q}^{s^k}(\tilde{r}^{k+1})^{k+1}.
        \end{split}
    \end{align*}
    Because $\tilde{r}^{k+1}$ is the minimizer of $\tilde{Q}^{s^k}(r)^{k+1}$, we have 
    \begin{align*}
        \begin{split}
            & \tilde{Q}^{s^k}(\tilde{r}^{k+1})^{k+1} \leq \tilde{Q}^{s^k}(r^k)^{k+1} = H^k(r^k,0),
        \end{split}
    \end{align*}
    which implies \cref{eq:d_relation8}.

    \begin{align*}
        \begin{split}
            \  &    D^{\tilde{s}^{k+1}}(r,\tilde{r}^{k+1}) 
                  - D^{\tilde{s}^k}(r,\tilde{r}^k) 
                  + D^{\tilde{s}^k}(\tilde{r}^{k+1},\tilde{r}^k)
        \\
          = \  &    J(r) 
                  - J(\tilde{r}^{k+1}) 
                  - \langle \tilde{s}^{k+1}, r - \tilde{r}^{k+1} \rangle 
        \\ 
            \  &    - J(r) 
                    + J(\tilde{r}^k) 
                    + \langle \tilde{s}^k,r-\tilde{r}^k \rangle 
        \\ 
            \  &    + J(\tilde{r}^{k+1}) 
                    - J(\tilde{r}^k) 
                    - \langle \tilde{s}^k, \tilde{r}^{k+1} - \tilde{r}^k \rangle 
        \\
           = \  &   \langle \tilde{s}^k - \tilde{s}^{k+1},  r - \tilde{r}^{k+1} \rangle 
        \\ 
           = \  &   \langle \tilde{s}^k - s^k + \tilde{q}^{k+1},  r - \tilde{r}^{k+1} \rangle
        \\ 
           = \  &   \langle \tilde{s}^k - s^k,  r - \tilde{r}^{k+1} \rangle + \langle \tilde{q}^{k+1},  r - \tilde{r}^{k+1} \rangle.
        \end{split}
    \end{align*}   
    Substituting $\tilde{s}^k = \eta^{k-1}(r^k_{ex} - \frac{\alpha}{\eta^{k-1}}\nabla \cdot \frac{\boldsymbol{\tau}^{\bot}}{\vert\boldsymbol{\tau}^{\bot}\vert})$ and $s^k = \eta^k(r^k_{ex} - \frac{\alpha}{\eta^k}\nabla \cdot \frac{\boldsymbol{\tau}^{\bot}}{\vert\boldsymbol{\tau}^{\bot}\vert})$ into the above transformation, we obtain
    \begin{align*}
        \begin{split}
            \  &    D^{\tilde{s}^{k+1}}(r,\tilde{r}^{k+1}) 
                  - D^{\tilde{s}^k}(r,\tilde{r}^k) 
                  + D^{\tilde{s}^k}(\tilde{r}^{k+1},\tilde{r}^k)
        \\ 
           = \  &   (\eta^k - \eta^{k-1})\langle r^k_{ex},  \tilde{r}^{k+1} - r \rangle + \langle \tilde{q}^{k+1},  r - \tilde{r}^{k+1} \rangle
        \\ 
           = \  &   \cfrac{\eta^k - \eta^{k-1}}{\eta^k} \langle r^k_{ex} - \tilde{r}^{k+1}, \tilde{q}^{k+1} \rangle + \cfrac{\eta^{k-1}}{\eta^k} \langle \tilde{q}^{k+1}, r - \tilde{r}^{k+1} \rangle + (\eta^k - \eta^{k-1}) \langle \tilde{r}^{k+1} - r^k_{ex}, r \rangle.
        \end{split}
    \end{align*}
    The $\tilde{q}^{k+1}$ is the subgradient of $H^k(\tilde{r}^{k+1},0)$. By the
    definition of subgradient, we have
    \begin{align*}
        \begin{split}
            &     D^{\tilde{s}^{k+1}}(r,\tilde{r}^{k+1}) 
                - D^{\tilde{s}^k}(r,\tilde{r}^k) 
                + D^{\tilde{s}^k}(\tilde{r}^{k+1},\tilde{r}^k)
                - (\eta^k - \eta^{k-1}) \langle \tilde{r}^{k+1} - r^k_{ex}, r \rangle
                  \\
       \leq &     \cfrac{\eta^k - \eta^{k-1}}{\eta^k} ( H^k(r^k_{ex},0) - H^k(\tilde{r}^{k+1},0) ) + \cfrac{\eta^{k-1}}{\eta^k} ( H^k(r,0) - H^k(\tilde{r}^{k+1},0) ) 
                  \\
          = &     \cfrac{\eta^k - \eta^{k-1}}{\eta^k} H^k(r^k_{ex},0) - H^k(\tilde{r}^{k+1},0).
        \end{split}
    \end{align*}
\end{proof}

The proposition (\ref{eq:d_relation8}) implies $\Vert\tilde{r}^{k+1}\Vert \le \Vert r^k \Vert$ corresponding to (\ref{eqs:tvs2m_k}) and (\ref{eqs:tvs2m_tk1}). By \cref{lemma:rof_fidelity_2}, we can obtain $\Vert r^{k+1} - r^k_{ex} \Vert \le \Vert\tilde{r}^{k+1} - r^k_{ex}\Vert$. With a Gaussian assumption, cf. \cref{lemma:gs}, we have  $\Vert r^{k+1} \Vert \le \Vert\tilde{r}^{k+1}\Vert$. The iteration series $\Vert r^i \Vert$, $i \in \mathbb{N}$, is therefore non-increase, i.e., $\Vert r^{i+1} \Vert \le \Vert r^i \Vert$.

If there exists a minimizer $r$ of $H(\cdot,0)$ with $J(r) < \infty$, by using (\ref{eq:d_relation9}), we have
\begin{align}
\label{eq:d_relation10}
    \begin{split}
            D^{\tilde{s}^{k+1}}(r,\tilde{r}^{k+1}) & \leq D^{\tilde{s}^{k+1}}(r,\tilde{r}^{k+1}) + D^{\tilde{s}^k}(\tilde{r}^{k+1},\tilde{r}^k) 
            \\ 
                                                   & \leq D^{\tilde{s}^{k+1}}(r,\tilde{r}^{k+1}) + D^{\tilde{s}^k}(\tilde{r}^{k+1},\tilde{r}^k) 
            \\
                                                   & \quad + H^k(\tilde{r}^{k+1},0) - (\eta^k - \eta^{k-1}) \langle \tilde{r}^{k+1} - r^k_{ex}, r \rangle
            \\
                                                   & \leq D^{\tilde{s}^k}(r,\tilde{r}^k) + \cfrac{\eta^k - \eta^{k-1}}{\eta^k} H^k(r^k_{ex},0). 
    \end{split}  
\end{align}
\begin{theorem}
\label{thm:convergence_2ndstp}
If $r \in BV(\Omega)$ is the minimizer of $H(\cdot,0)$ subject to $k 
\in \mathbb{N}$, then $r^k$ converges
and
\begin{equation}\label{eq:h_relation2}
    \Vert r^k \Vert^2 \leq \cfrac{2 D^{\tilde{s}^1}(r,\tilde{r}^1) + 2 \beta \Vert r^0_{ex} \Vert_{\infty}}{k\eta^1},
\end{equation}
moreover, 
\begin{equation*}
    u^k = \sum_{i=1}^{k}
    r^{i},
\end{equation*} 
converges to $f$.
\end{theorem}
\begin{proof}
Taking the sum of (\ref{eq:d_relation9}) as follows
\begin{align*}
    \begin{split}
        & \sum_{i=1}^k \left[ D^{\tilde{s}^{i+1}}(r,\tilde{r}^{i+1}) 
                - D^{\tilde{s}^i}(r,\tilde{r}^i) 
                + D^{\tilde{s}^i}(\tilde{r}^{i+1},\tilde{r}^i)
                + H^i(\tilde{r}^{i+1},0)
                - \cfrac{\eta^i - \eta^{i-1}}{\eta^i} H^i(r^i_{ex},0) \right],
    \end{split}
\end{align*}
we obtain
\begin{align}
\label{eq:d_relation11}
    \begin{split}
        D^{\tilde{s}^{k+1}}(r,\tilde{r}^{k+1}) + 
        \sum_{i=1}^k 
        \left[
            D^{\tilde{s}^i}(\tilde{r}^{i+1},\tilde{r}^i) + 
            H^i(\tilde{r}^{i+1},0) -
            \cfrac{\eta^i - \eta^{i-1}}{\eta^i} H^i(r^i_{ex},0)
        \right] &
        \\
        \leq  D^{\tilde{s}^1}(r,\tilde{r}^1). & 
    \end{split}
\end{align}
Due to the non-negativity of Bregmann distance, the above inequality can be rewritten as follows
\begin{align}
\label{eq:d_relation12}
    \begin{split}
        & \sum_{i=1}^k 
        \left[
            H^i(\tilde{r}^{i+1},0) -
            \cfrac{\eta^i - \eta^{i-1}}{\eta^i} H^i(r^i_{ex},0)
        \right] 
        \leq  D^{\tilde{s}^1}(r,\tilde{r}^1). 
        \end{split}
\end{align}
Substituting $\eta^i = \frac{\beta}{\gamma^i} \sim \frac{2 \beta}{\Vert r^i_{ex} \Vert_{\infty}}$ into the above inequality, we obtain
\begin{align*}
    \begin{split}
        & \sum_{i=1}^k 
        \left[
            H^i(\tilde{r}^{i+1},0) -
            \cfrac{\eta^i - \eta^{i-1}}{2}\Vert r^i_{ex} \Vert^2
        \right] 
        \leq  D^{\tilde{s}^1}(r,\tilde{r}^1),
        \\
        \Rightarrow & \sum_{i=1}^k 
        \left[
            H^i(\tilde{r}^{i+1},0) -
            \cfrac{\beta(\Vert r^{i-1}_{ex} \Vert_{\infty} - \Vert r^i_{ex} \Vert_{\infty})}{\Vert r^i_{ex} \Vert_{\infty} \Vert r^{i-1}_{ex} \Vert_{\infty}}\Vert r^i_{ex} \Vert^2
        \right] 
        \leq  D^{\tilde{s}^1}(r,\tilde{r}^1),
        \\
        \Rightarrow & \sum_{i=1}^k 
        \left[
            \cfrac{\eta^i}{2} \Vert\tilde{r}^{i+1}\Vert^2 -
            \beta(\Vert r^{i-1}_{ex} \Vert_{\infty} - \Vert r^i_{ex} \Vert_{\infty})
        \right] 
        \leq  D^{\tilde{s}^1}(r,\tilde{r}^1).
        \end{split}
\end{align*}
Since $\Vert\tilde{r}^k\Vert$ is monotonically nonincreasing, it results
\begin{equation*}
        k \eta^1 \Vert\tilde{r}^k\Vert^2
        \leq 2 D^{\tilde{s}^1}(r,\tilde{r}^1) + 2 \beta(\Vert r^0_{ex} \Vert_{\infty} - \Vert r^k_{ex} \Vert_{\infty})
        \leq 2 D^{\tilde{s}^1}(r,\tilde{r}^1) + 2 \beta \Vert r^0_{ex} \Vert_{\infty}.  
\end{equation*}
We obtain \cref{eq:h_relation2}. It implies that, when $k \rightarrow \infty$, $r^k$ converges to $0$ 
with rate
\begin{equation*}
    \Vert r^k \Vert 
    \leq \Vert\tilde{r}^k\Vert
    \leq \sqrt{ \frac{2 D^{\tilde{s}^1}(r,\tilde{r}^1) + 2 \beta \Vert r^0_{ex} \Vert_{\infty}}{k\eta^1} } 
    = \mathcal{O}(k^{-1/2}). 
\end{equation*}

When $r^k$ converges to $0$, it is easy to find $u^k$ converges to $f$ by making a contradiction against Lemma (\ref{lemma:fidelity_constraint}).

\end{proof}

\begin{remark}
In the proof of \cref{thm:convergence_2ndstp}, defining $\eta^i \sim \frac{2\beta}{\Vert r^i_{ex}\Vert_{\infty}}$, only the case $\eta^{i-1}<=\eta^i$ is considered, otherwise, as our setting $\eta^i = \max ( \beta / \gamma, \eta^{i-1})$, $\eta^{i-1}=\eta^i$ so that the convergence follows directly by the fact that the term $\cfrac{\eta^i - \eta^{i-1}}{\eta^i} H^i(r^i_{ex},0) $ vanishes in (\ref{eq:d_relation12}). 

\end{remark}

\subsection{Iterative regularization applied separately to each of TV-Stokes steps} 
The Richardson-like iterations are applied separately on both two steps of TV-Stokes model is listed below, cf. \Cref{alg:tvs12}. The properties of well-definedness and convergence naturally follow from the separated cases addressed in previous subsections.

\begin{algorithm}
    \caption{Separated iterative regularization for TV-Stokes}
    \label{alg:tvs12}
    \begin{algorithmic}[1]
        \State{Initialize $k = 0$, $\mathbf{r}_{ex}^0 = \nabla^{\bot} f$, $\boldsymbol{\tau}
        = 0$;}
        \Repeat
            \State{$k = k + 1$;}        
            \State{$\mathbf{r}^{k} = \mbox{arg} \min_{\mathbf{r} \in BV(\Omega)} \{ 
            J(\Pi\mathbf{r}) + H(\mathbf{r},\mathbf{r}_{ex}^{k-1}) \}$;}
            \State{$\mathbf{r}_{ex}^{k}  = \mathbf{r}_{ex}^{k-1} - \mathbf{r}^{k}$;}   
            \State{$\boldsymbol{\tau} = \boldsymbol{\tau} + \mathbf{r}^k$;}
        \Until{satisfied}
        \State{Initialize $k = 0$, $r_{ex}^0 := f$, $u = 0$;}
        \Repeat
            \State{$k = k + 1$;}        
            \State{$r^k = \mbox{arg} \min_{r \in BV(\Omega)}  \{ J(r)  - \alpha \nabla \cdot \frac{\boldsymbol{\tau}^{\bot}}{\vert\boldsymbol{\tau}^{\bot}\vert} + 
            H^k(r,r_{ex}^{k-1}) \}$;}
            \State{$r_{ex}^{k}  = r_{ex}^{k-1} - r^{k} $;}
            \State{$u = u + r^k$;}
        \Until{satisfied}
        \State \textbf{return} $u$.
    \end{algorithmic} 
\end{algorithm}

\section{Numerical experiments}
\label{sec:numt}

In this section, we present our experiments on the effectiveness of proposed algorithms on smooth structures, e.g., Lena's face, on their capability in preserving both sharp edges and smooth patterns, e.g., fingerprint with clean surrounding, and finally on structures mixed with pinstripes and smooth surfaces, cf. e.g., figures of Barbara.  

In our experiments, we employ the dual-formula-based method to solve the TV-Stokes model, cf. \cite{Elo2009a}, where we keep the step sizes for the line search the same as $1/4$ throughout the experiments. The noise resource considered in this paper is of Gaussian.

\begin{figure}[!htbp]
    \captionsetup[subfigure]{
                            justification=centering,
                            labelformat=empty,
                            }
    \centering
    \begin{minipage}{1.0\textwidth}
        \centering
        \subfloat[][Noisy image]
            {
            \includegraphics[width=0.3\linewidth,
                            height=0.3\linewidth,
                            angle=0]{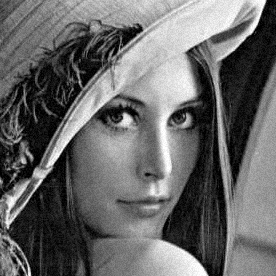}
            }
            \hspace{0.0005\linewidth}
        \subfloat[][Restored image]
            {
            \includegraphics[width=0.3\linewidth,
                            height=0.3\linewidth,
                            angle=0]{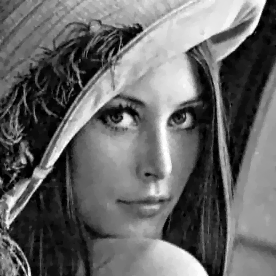}
            }
            \hspace{0.0005\linewidth}
        \subfloat[][Clean image]
            {
            \includegraphics[width=0.3\linewidth,
                            height=0.3\linewidth,
                            angle=0]{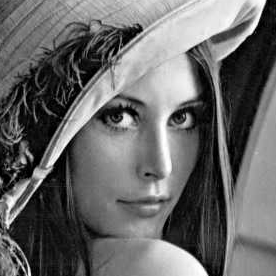}
            }
    \end{minipage}
    \hfill
    \vspace{0.005\linewidth}
    \begin{minipage}{1.0\textwidth}
        \centering
        \subfloat[][Initial Gaussian noise]
            {
            \includegraphics[width=0.3\linewidth,
                            height=0.3\linewidth,
                            angle=0]{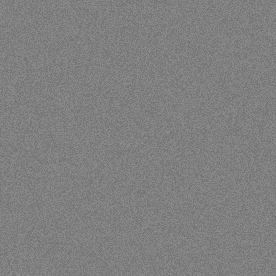}
            }
            \hspace{0.0005\linewidth}
        \subfloat[][Noise: Restored image]
            {
            \includegraphics[width=0.3\linewidth,
                            height=0.3\linewidth,
                            angle=0]{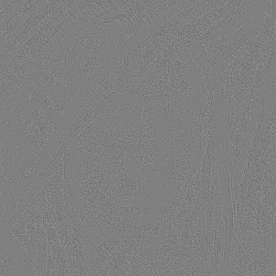}
            }
            \hspace{0.0005\linewidth}
        \subfloat[][$\Vert u - g \Vert$]
            {
                \begin{tikzpicture}
                    \node[anchor=south west,inner sep=0pt,outer sep=0pt] (img)
                    {
                            \includegraphics[width=0.3\linewidth, 
                                            height=0.3\linewidth,
                                            angle=0
                                            ]{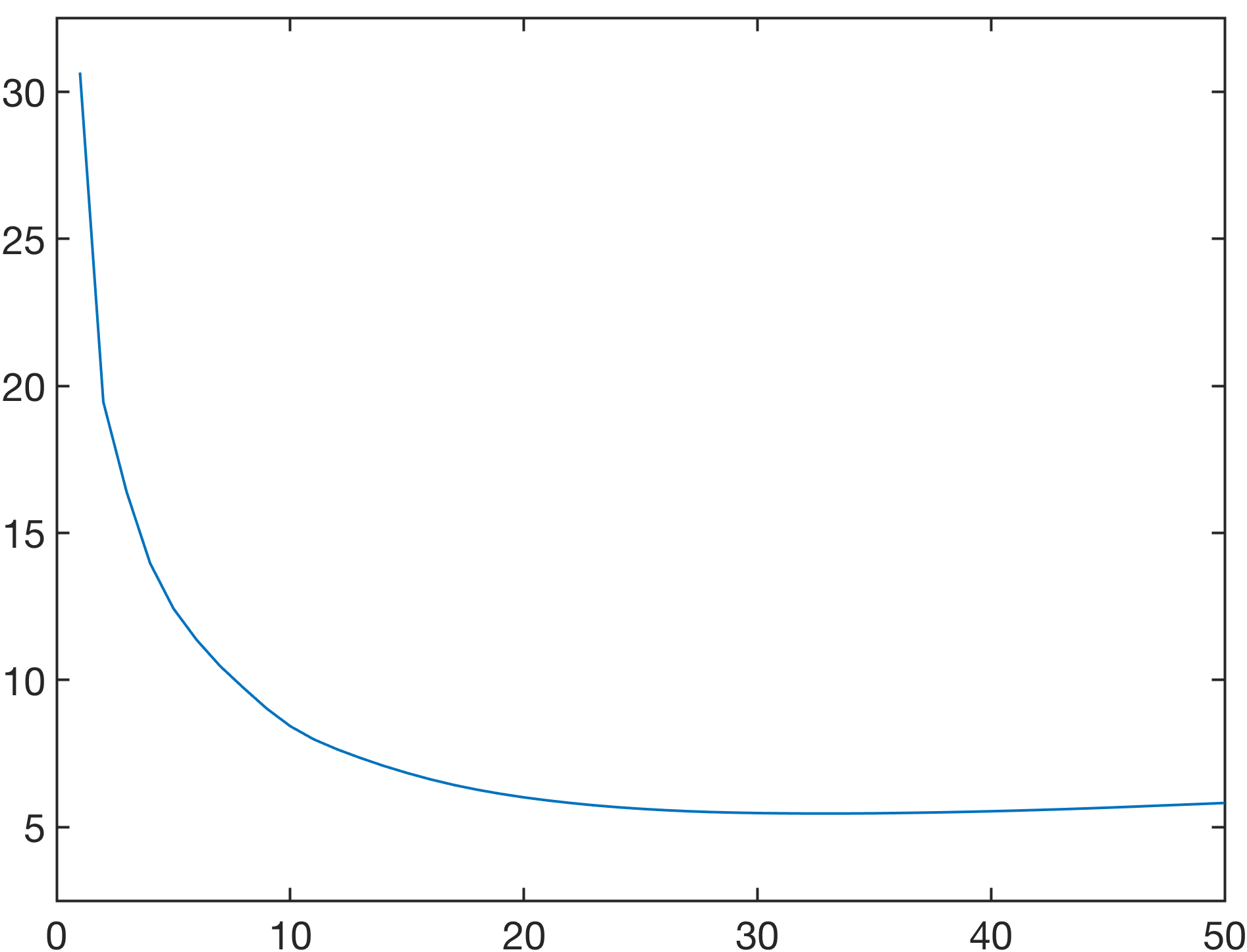}
                    };
                    \begin{scope}[x=(img.south east),y=(img.north west)]
                        \node[circle,draw,red!80,inner sep=0pt,outer sep=0pt,minimum size=0.2cm] (P1) at (0.65,0.15) {};
                    \end{scope}
                \end{tikzpicture}
            }
    \end{minipage}
    \hfill
    \vspace{0.005\linewidth}
    \begin{minipage}{1.0\textwidth}
        \centering
        \subfloat[][$\Vert u - f \Vert$]
            {
                \begin{tikzpicture}
                    \node[anchor=south west,inner sep=0pt,outer sep=0pt] (iimg)
                    {
                            \includegraphics[width=0.3\linewidth, 
                                            height=0.3\linewidth,
                                            angle=0
                                            ]{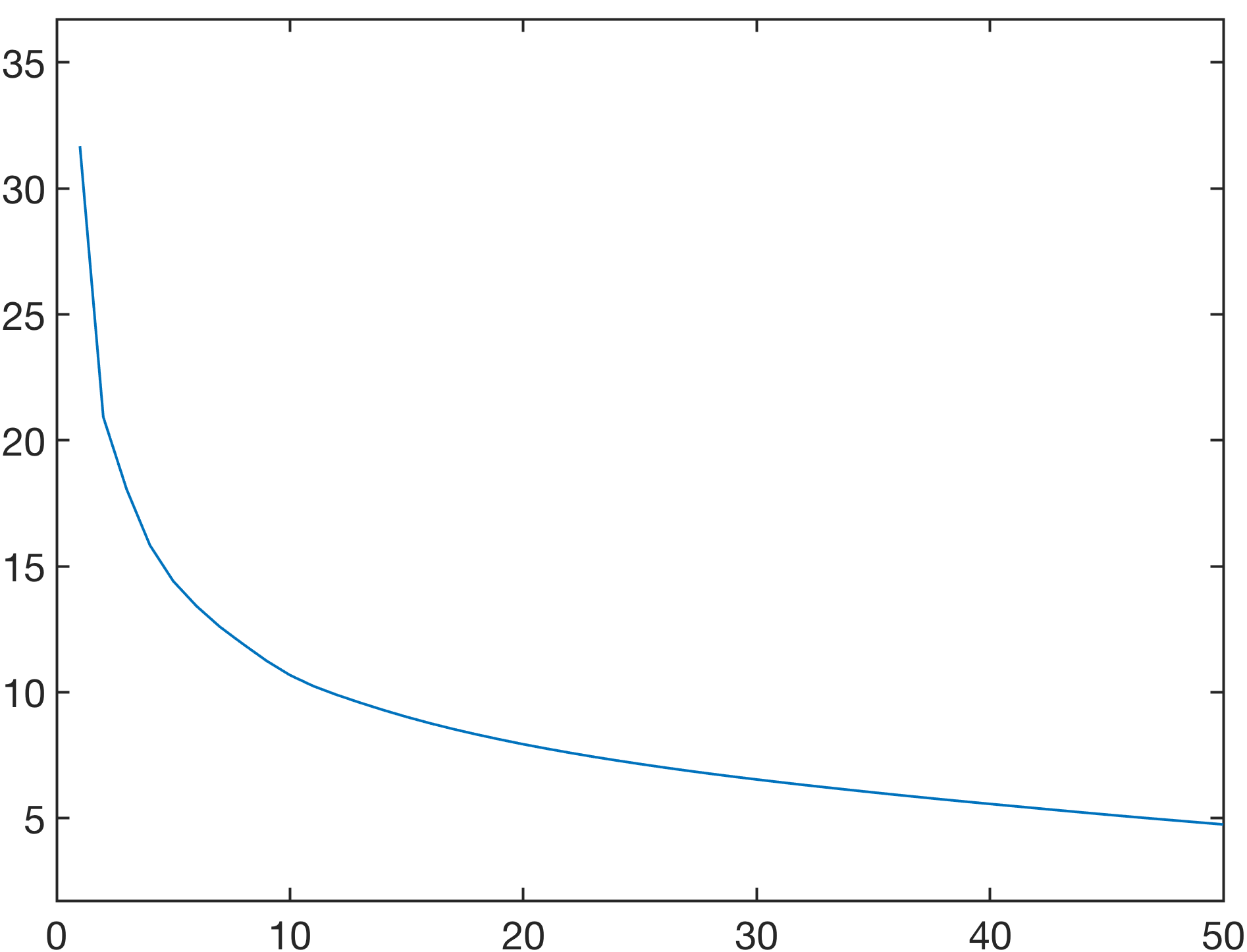}
                    };
                    \begin{scope}[x=(iimg.south east),y=(iimg.north west)]
                        \node[circle,draw,red!80,inner sep=0pt,outer sep=0pt,minimum size=0.2cm] (P2) at (0.65,0.17) {};
                        \draw[dashed,red] (0.05,0.20) -- (0.98,0.20);
                    \end{scope}
                \end{tikzpicture}
            }
            \hspace{0.0005\linewidth}
        \subfloat[][]
            {
                {\transparent{0.0}
                    \includegraphics[width=0.2828\linewidth,
                                    height=0.3\linewidth,
                                    angle=0]{lena_noisy_1_m10_ns2}
                }
            }
            \hspace{0.0005\linewidth}
        \subfloat[][]
            {
                {\transparent{0.0}
                    \includegraphics[width=0.2828\linewidth,
                                    height=0.3\linewidth,
                                    angle=0]{lena_noisy_1_m10_ns2}
                }
            }
    \end{minipage}
    \caption{Showing the effect of applying the Osher-like iterative regularization on the ROF model, cf. \Cref{alg:osher}. The initial image is of a Gaussian noise at noise level $7.97$ and with $PSNR = 30.79$. The final restored image is at noise level $5.46$ and with $PSNR = 34.08$.}
    \label{fig:alm10}
\end{figure}
We start the experiment with \Cref{alg:osher}, applying Osher's iterative regularization algorithm on a Lena portrait, cf. \cref{fig:alm10}. The associated ROF model is solved via the Chambolle dual formula, cf. \cite{Chambolle1997} for the details. The initial noise level is $7.97$ while the corresponding Peak Signal-Noise Ratio (PSNR) is $30.79$. In this paper, the PSNR number is calculated via the \textsc{Matlab} function \verb|psnr|. The curve $\Vert u-g \Vert$ shows an optimal solution at the iteration $33$ where the restored image $u$ is most close to the clean image $g$ in $L^2$, the curve $\Vert u-f \Vert$ shows the resulted image $u$ through the iteration is converging to the initial image $f$. The restored image via this experiment is at noise level $5.46$ and with PSNR $34.08$. The result suffers the effect of stair-case inherited from ROF model.

For verifying the effectiveness of our proposed algorithms, we apply \Cref{alg:tvs1}, \Cref{alg:tvs2}, and \Cref{alg:tvs12}, respectively, on the same image, the Lena portrait, with the same noise. 

\begin{figure}[!htbp]
    \captionsetup[subfigure]{
                            justification=centering,
                            labelformat=empty,
                            }
    \centering
    \begin{minipage}{1.0\textwidth}
        \centering
        \subfloat[][Noisy image]
            {
            \includegraphics[width=0.3\linewidth,
                            height=0.3\linewidth,
                            angle=0]{lena_noisy_1}
            }
            \hspace{0.0005\linewidth}
        \subfloat[][Restored image]
            {
            \includegraphics[width=0.3\linewidth,
                            height=0.3\linewidth,
                            angle=0]{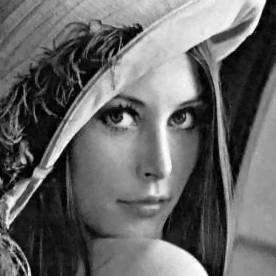}
            }
            \hspace{0.0005\linewidth}
        \subfloat[][Clean image]
            {
            \includegraphics[width=0.3\linewidth,
                            height=0.3\linewidth,
                            angle=0]{lena_clean}
            }
    \end{minipage}
    \hfill
    \vspace{0.005\linewidth}
    \begin{minipage}{1.0\textwidth}
        \centering
        \subfloat[][Initial Gaussian noise]
            {
            \includegraphics[width=0.3\linewidth,
                            height=0.3\linewidth,
                            angle=0]{lena_noisy_1_noise}
            }
            \hspace{0.0005\linewidth}
        \subfloat[][Noise: Restored image]
            {
            \includegraphics[width=0.3\linewidth,
                            height=0.3\linewidth,
                            angle=0]{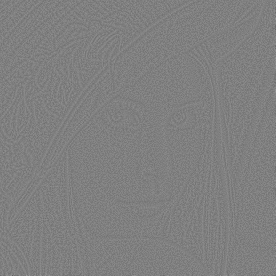}
            }
            \hspace{0.0005\linewidth}
        \subfloat[][]
            {
                {\transparent{0.0}
                    \includegraphics[width=0.2828\linewidth,
                                    height=0.3\linewidth,
                                    angle=0]{lena_noisy_1_m1_noise}
                }
            }
    \end{minipage}
    \caption{Showing the effect of applying the iterative regularization on the first step of TV-Stokes model, cf. \Cref{alg:tvs1}. The initial image is of a Gaussian noise at noise level $7.97$ and with $PSNR = 30.79$. The final restored image is of noise level $5.10$ and $PSNR = 34.66$.}
    \label{fig:alm1}
\end{figure}
We first apply \Cref{alg:tvs1} on the Lena portrait, that is using the Richardson-like iterative regularization on the first step of the TV-Stokes model, cf. \cref{fig:alm1}. The parameter $\beta$ is set to be $6.5$ and $3$ for the first step and the second step, respectively. The parameter $\alpha$ is given to be $0.9$. We observe that the optimal solution for the first step achieves at the iteration $13$ in our experiment. The final restored image is at noise level $5.10$ and with PSNR $34.66$. The result shows an improvement of the smoothness of Lena's face.

\begin{figure}[!htbp]
    \captionsetup[subfigure]{
                            justification=centering,
                            labelformat=empty,
                            }
    \centering
    \begin{minipage}{0.98\textwidth}
        \centering
        \subfloat[][Noisy image]
            {
                    \includegraphics[width=0.3\linewidth, 
                                    height=0.3\linewidth,
                                    angle=0
                                    ]{lena_noisy_1}
            }
            \hspace{0.0005\linewidth}
        \subfloat[][The first iteration]
            {
                    \includegraphics[width=0.3\linewidth, 
                                    height=0.3\linewidth,
                                    angle=0
                                    ]{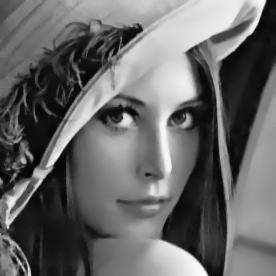}
            }
            \hspace{0.0005\linewidth}
        \subfloat[][The third iteration]
            {
                    \includegraphics[width=0.3\linewidth, 
                                    height=0.3\linewidth,
                                    angle=0
                                    ]{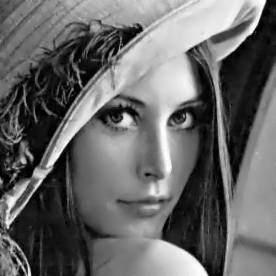}
            }
    \end{minipage}
    \hfill
    \vspace{0.005\linewidth}
    \begin{minipage}{0.98\textwidth}
        \centering
        \subfloat[][Gaussian noise]
            {
                    \includegraphics[width=0.3\linewidth, 
                                    height=0.3\linewidth,
                                    angle=0
                                    ]{lena_noisy_1_noise}
            }
            \hspace{0.0005\linewidth}
        \subfloat[][Noise: The first iteration]
            {
                    \includegraphics[width=0.3\linewidth, 
                                    height=0.3\linewidth,
                                    angle=0
                                    ]{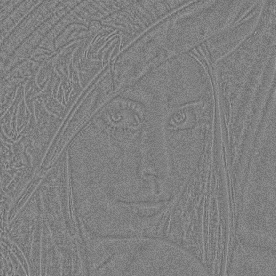}
            }
            \hspace{0.0005\linewidth}
        \subfloat[][Noise: The third iteration]
            {
                    \includegraphics[width=0.3\linewidth, 
                                    height=0.3\linewidth,
                                    angle=0
                                    ]{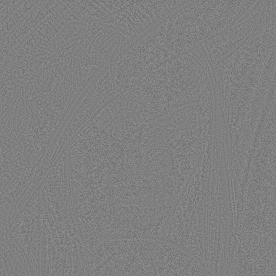}
            }
    \end{minipage}
    \hfill
    \vspace{0.005\linewidth}
    \begin{minipage}{0.98\textwidth}
        \centering
        \subfloat[][The thirteenth iteration]
            {
                    \includegraphics[width=0.3\linewidth, 
                                    height=0.3\linewidth,
                                    angle=0
                                    ]{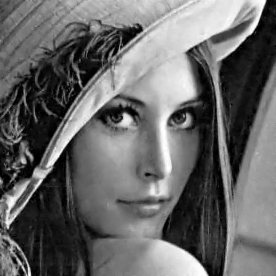}
            }
            \hspace{0.0005\linewidth}
        \subfloat[][Clean image]
            {
                    \includegraphics[width=0.3\linewidth, 
                                    height=0.3\linewidth,
                                    angle=0
                                    ]{lena_clean}
            }
            \hspace{0.0005\linewidth}
        \subfloat[][$\Vert u - g \Vert$]
            {
                \begin{tikzpicture}
                    \node[anchor=south west,inner sep=0pt,outer sep=0pt] (img)
                    {
                            \includegraphics[width=0.3\linewidth, 
                                            height=0.3\linewidth,
                                            angle=0
                                            ]{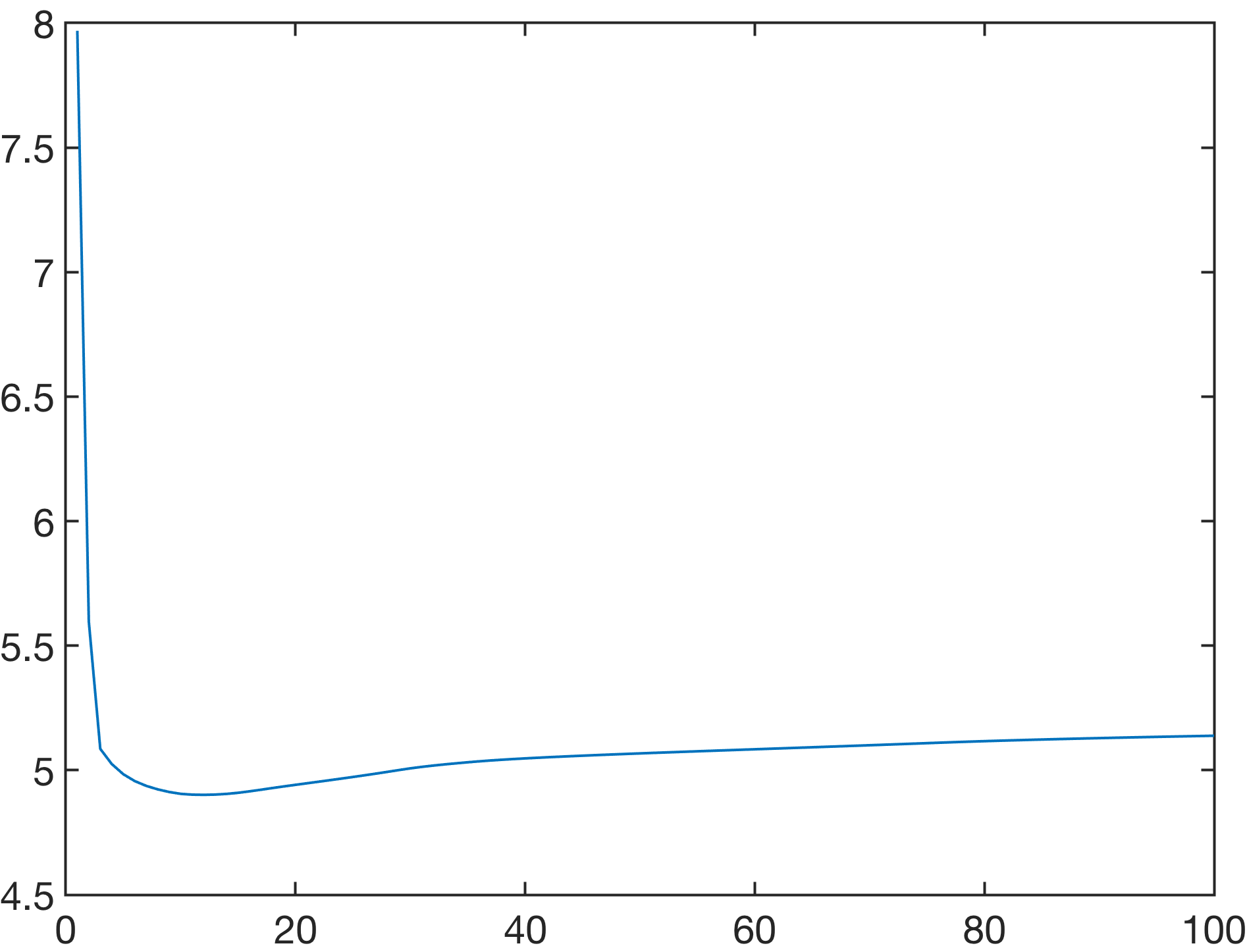}
                    };
                    \begin{scope}[x=(img.south east),y=(img.north west)]
                        \node[circle,draw,red!80,inner sep=0pt,outer sep=0pt,minimum size=0.2cm] (P1) at (0.15,0.17) {};
                    \end{scope}
                \end{tikzpicture}
            }
    \end{minipage}
    \hfill
    \vspace{0.005\linewidth}
    \begin{minipage}{0.98\textwidth}
        \centering
        \subfloat[][Noise: The thirteenth iteration]
            {
                    \includegraphics[width=0.3\linewidth, 
                                    height=0.3\linewidth,
                                    angle=0
                                    ]{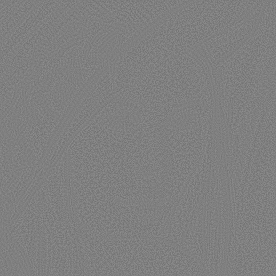}
            }
            \hspace{0.0005\linewidth}
        \subfloat[][$\Vert u - f \Vert$]
            {
                \begin{tikzpicture}
                    \node[anchor=south west,inner sep=0pt,outer sep=0pt] (iimg)
                    {
                            \includegraphics[width=0.3\linewidth, 
                                            height=0.3\linewidth,
                                            angle=0
                                            ]{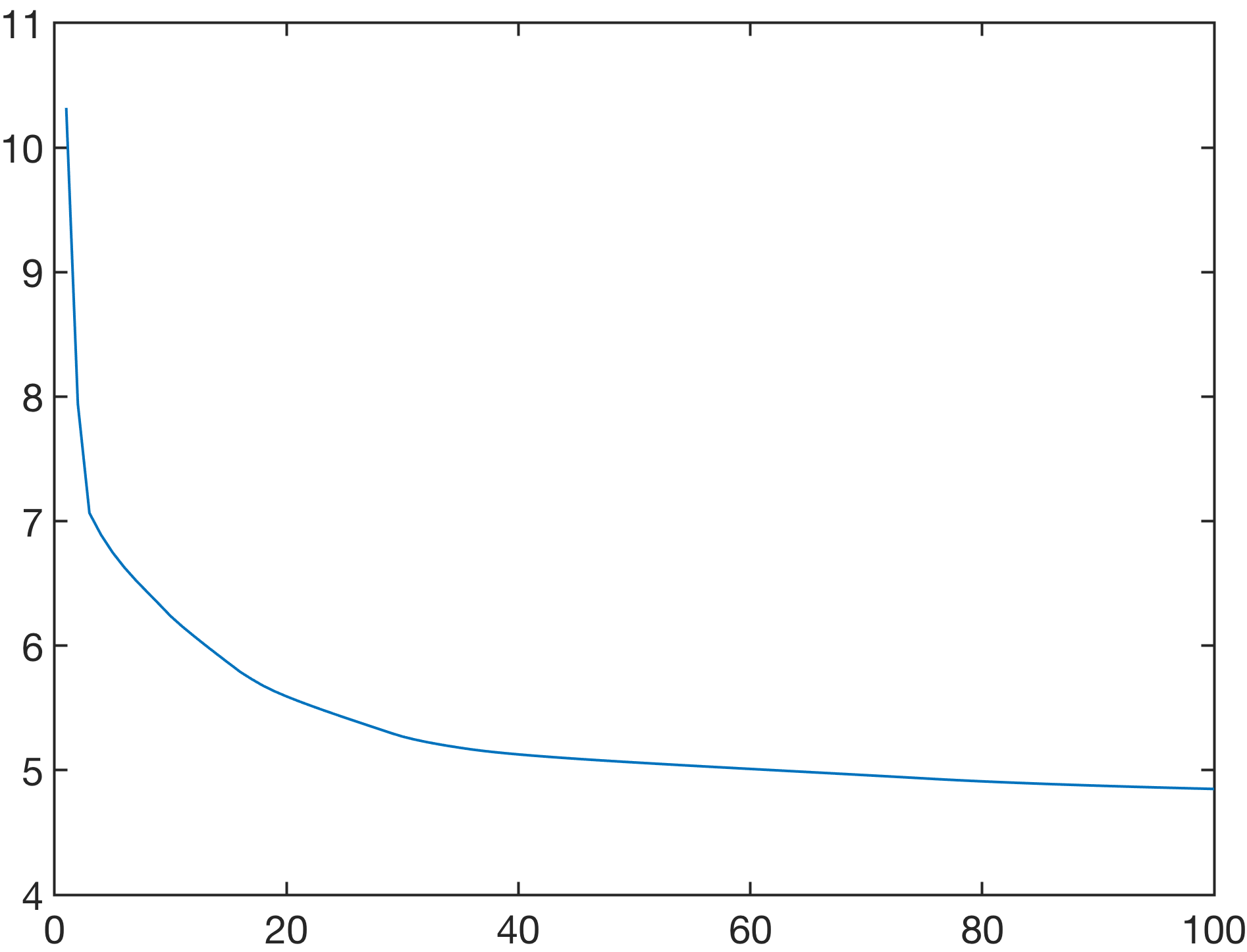}
                    };
                    \begin{scope}[x=(iimg.south east),y=(iimg.north west)]
                        \node[circle,draw,red!80,inner sep=0pt,outer sep=0pt,minimum size=0.2cm] (P2) at (0.15,0.335) {};
                        \draw[dashed,red] (0.05,0.58) -- (0.98,0.58);
                    \end{scope}
                \end{tikzpicture}
            }
            \hspace{0.0005\linewidth}
        \subfloat[][]
            {
                {\transparent{0.0}
                    \includegraphics[width=0.2828\linewidth, 
                                    height=0.3\linewidth,
                                    angle=0
                                    ]{lena_noisy_1_m2_ns2}
                }
            }
    \end{minipage}
    \caption{Showing the effect of applying the iterative regularization on the second step of TV-Stokes model, cf. \Cref{alg:tvs2}. The initial image is of a Gaussian noise at noise level $7.97$ and with $PSNR = 30.79$. The final restored image is of noise level $4.90$ and $PSNR = 35.01$.}
    \label{fig:alm2}
\end{figure}
The next experiment is applying the \Cref{alg:tvs2} on the same Lena portrait, that is using the Richardson-like iterative regularization on the second step of the TV-Stokes model, cf. \cref{fig:alm2}. The parameter $\beta$ is set to be $8.0$ and $2.5$ for the first step and the second step, respectively. The parameter $\alpha$ is given to be $0.9$. We observe that the optimal solution for the second step achieves at the iteration $12$ in our experiment. The curve $\Vert u-g \Vert$ shows an optimal solution at the iteration $12$ where the restored image $u$ is most close to the clean image $g$ in $L^2$, the curve $\Vert u-f \Vert$ shows the resulted image $u$ through the iteration is converging to the initial image $f$. The final restored image is at noise level $4.90$ and with PSNR $35.01$. The result shows a visible improvement both in the smoothness of Lena's face and in the preserving of details of the hat, comparing to \cref{fig:alm10} via Osher-like iterative regularization. The textures and fine structures are observed to be added back to the restored image accumulatively through the iterations. 

\begin{figure}[!htbp]
    \captionsetup[subfigure]{
                            justification=centering,
                            labelformat=empty,
                            }
    \centering
    \begin{minipage}{1.0\textwidth}
        \centering
        \subfloat[][Noisy image]
            {
            \includegraphics[width=0.3\linewidth,
                            height=0.3\linewidth,
                            angle=0]{lena_noisy_1}
            }
            \hspace{0.0005\linewidth}
        \subfloat[][Restored image]
            {
            \includegraphics[width=0.3\linewidth,
                            height=0.3\linewidth,
                            angle=0]{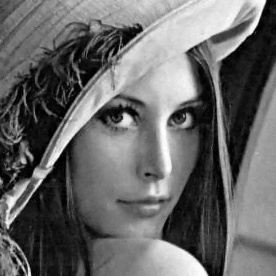}
            }
            \hspace{0.0005\linewidth}
        \subfloat[][Clean image]
            {
            \includegraphics[width=0.3\linewidth,
                            height=0.3\linewidth,
                            angle=0]{lena_clean}
            }
    \end{minipage}
    \hfill
    \vspace{0.005\linewidth}
    \begin{minipage}{1.0\textwidth}
        \centering
        \subfloat[][Gaussian noise]
            {
            \includegraphics[width=0.3\linewidth,
                            height=0.3\linewidth,
                            angle=0]{lena_noisy_1_noise}
            }
            \hspace{0.0005\linewidth}
        \subfloat[][Noise: Restored image]
            {
            \includegraphics[width=0.3\linewidth,
                            height=0.3\linewidth,
                            angle=0]{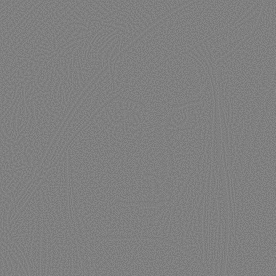}
            }
            \hspace{0.0005\linewidth}
        \subfloat[][$\Vert u - g \Vert$]
            {
                \begin{tikzpicture}
                    \node[anchor=south west,inner sep=0pt,outer sep=0pt] (img)
                    {
                            \includegraphics[width=0.3\linewidth, 
                                            height=0.3\linewidth,
                                            angle=0
                                            ]{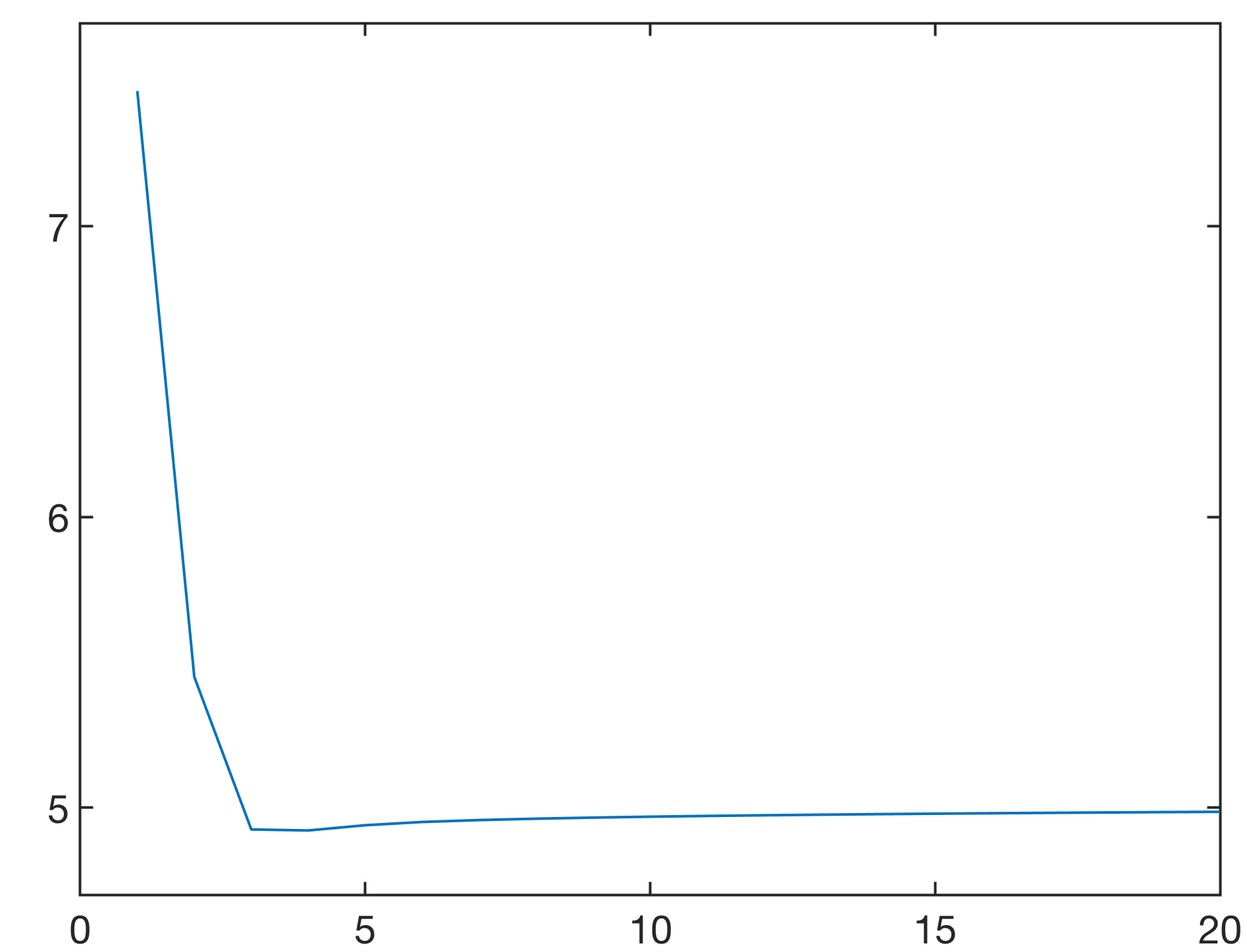}
                    };
                    \begin{scope}[x=(img.south east),y=(img.north west)]
                        \node[circle,draw,red!80,inner sep=0pt,outer sep=0pt,minimum size=0.2cm] (P1) at (0.20,0.13) {};
                    \end{scope}
                \end{tikzpicture}
            }
    \end{minipage}
    \hfill
    \vspace{0.005\linewidth}
    \begin{minipage}{1.0\textwidth}
        \centering
        \subfloat[][$\Vert u - f \Vert$]
            {
                \begin{tikzpicture}
                    \node[anchor=south west,inner sep=0pt,outer sep=0pt] (iimg)
                    {
                            \includegraphics[width=0.3\linewidth, 
                                            height=0.3\linewidth,
                                            angle=0
                                            ]{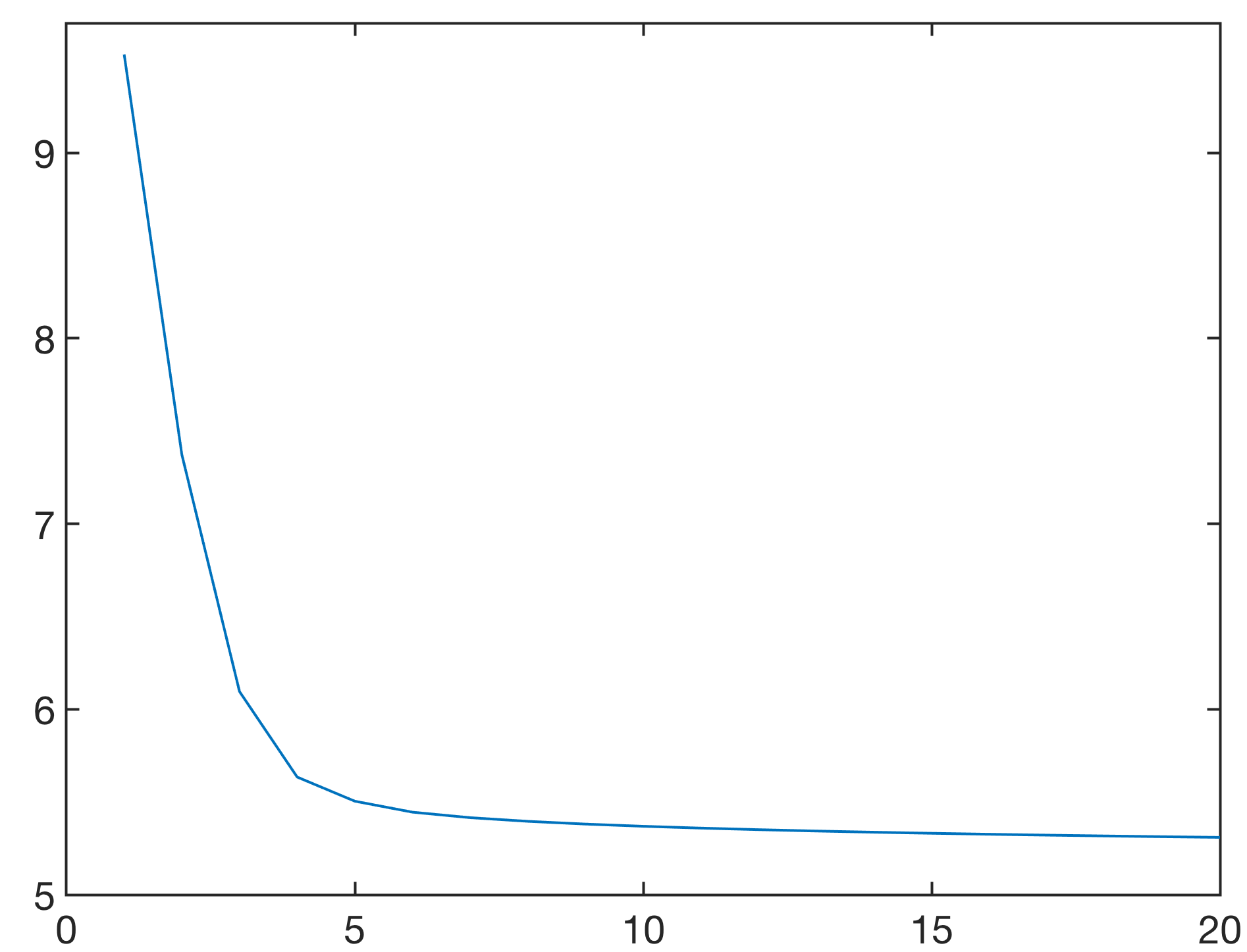}
                    };
                    \begin{scope}[x=(iimg.south east),y=(iimg.north west)]
                        \node[circle,draw,red!80,inner sep=0pt,outer sep=0pt,minimum size=0.2cm] (P2) at (0.20,0.27) {};
                        \draw[dashed,red] (0.05,0.64) -- (0.98,0.64);
                    \end{scope}
                \end{tikzpicture}
            }
            \hspace{0.0005\linewidth}
        \subfloat[][]
            {
                {\transparent{0.0}
                    \includegraphics[width=0.2828\linewidth,
                                    height=0.3\linewidth,
                                    angle=0]{lena_noisy_1_m3_ns4}
                }
            }
            \hspace{0.0005\linewidth}
        \subfloat[][]
            {
                {\transparent{0.0}
                    \includegraphics[width=0.2828\linewidth,
                                    height=0.3\linewidth,
                                    angle=0]{lena_noisy_1_m3_ns4}
                }
            }
    \end{minipage}
    \caption{Showing the effect of applying the iterative regularization onto both the first step and the second step of TV-Stokes model, cf. \Cref{alg:tvs12}. The initial image is of a Gaussian noise at noise level $7.97$ and with $PSNR = 30.79$. The final restored image is of noise level $4.94$ and $PSNR = 34.95$.}
    \label{fig:alm3}
\end{figure}
The last experiment on the same noisy Lena portrait is applying the \Cref{alg:tvs12}, that is using the Richardson-like iterative regularization on both the first step and the second step of the TV-Stokes model, cf. \cref{fig:alm3}. The parameter $\beta$ is set to be $6.5$ and $2.5$ for the first step and the second step, respectively. The parameter $\alpha$ is given to be $0.9$. We observe that the optimal solution for the first step achieves at the iteration $13$ and for the second step achieves at the iteration $3$ in our experiment. The curve $\Vert u-g \Vert$ shows an optimal solution at the iteration $3$ where the restored image $u$ is most close to the clean image $g$ in $L^2$, the curve $\Vert u-f \Vert$ shows the resulted image $u$ through the iteration is converging to the initial image $f$. The final restored image is at noise level $4.94$ and with PSNR $34.95$. The result also shows a visible improvement both in the smoothness of Lena's face and in the preserving of details of the hat, comparing to \cref{fig:alm10} via Osher-like iterative regularization.  

\begin{figure}[!htbp]
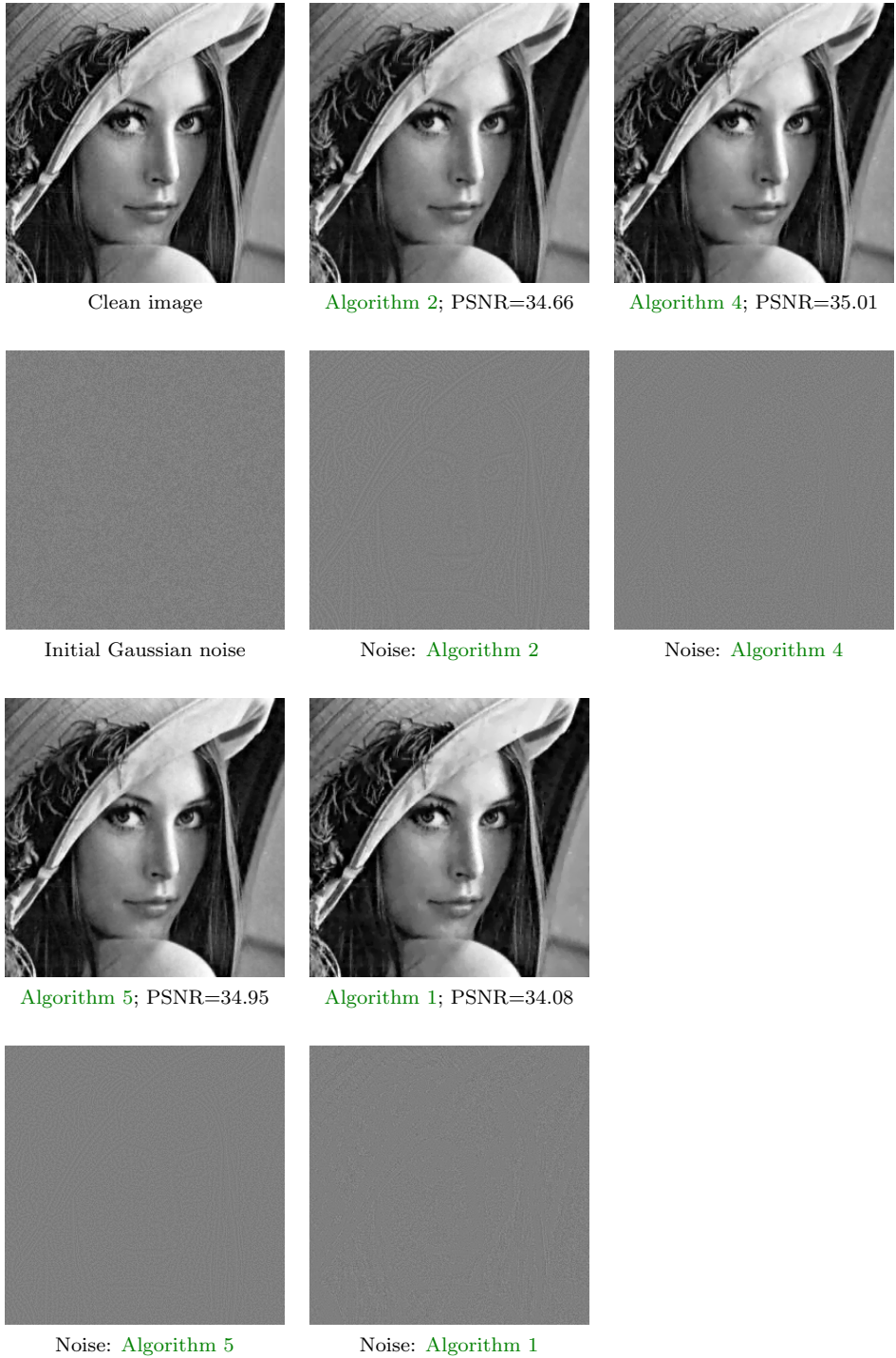

    \captionsetup[subfigure]{
                            justification=centering,
                            labelformat=empty,
                            }
    \centering
    \begin{minipage}{1.0\textwidth}
        \centering
        \subfloat[][Clean image]
            {
            \includegraphics[width=0.3\linewidth,
                            height=0.3\linewidth,
                            angle=0]{lena_clean}
            }
            \hspace{0.0005\linewidth}
        \subfloat[][\Cref{alg:tvs1}; PSNR=34.66]
            {
            \includegraphics[width=0.3\linewidth,
                            height=0.3\linewidth,
                            angle=0]{lena_noisy_1_m1}
            }
            \hspace{0.0005\linewidth}
        \subfloat[][\Cref{alg:tvs2}; PSNR=35.01]
            {
            \includegraphics[width=0.3\linewidth,
                            height=0.3\linewidth,
                            angle=0]{lena_noisy_1_m2}
            }
    \end{minipage}
    \hfill
    \vspace{0.005\linewidth}
    \begin{minipage}{1.0\textwidth}
        \centering
        \subfloat[][Initial Gaussian noise]
            {
            \includegraphics[width=0.3\linewidth,
                            height=0.3\linewidth,
                            angle=0]{lena_noisy_1_noise}
            }
            \hspace{0.0005\linewidth}
        \subfloat[][Noise: \Cref{alg:tvs1}]
            {
            \includegraphics[width=0.3\linewidth,
                            height=0.3\linewidth,
                            angle=0]{lena_noisy_1_m1_noise}
            }
            \hspace{0.0005\linewidth}
        \subfloat[][Noise: \Cref{alg:tvs2}]
            {
            \includegraphics[width=0.3\linewidth,
                            height=0.3\linewidth,
                            angle=0]{lena_noisy_1_m2_noise}
            }
    \end{minipage}
    \hfill
    \vspace{0.005\linewidth}
    \begin{minipage}{1.0\textwidth}
        \centering
        \subfloat[][\Cref{alg:tvs12}; PSNR=34.95]
            {
            \includegraphics[width=0.3\linewidth,
                            height=0.3\linewidth,
                            angle=0]{lena_noisy_1_m3}
            }
            \hspace{0.0005\linewidth}
        \subfloat[][\Cref{alg:osher}; PSNR=34.08]
            {
            \includegraphics[width=0.3\linewidth,
                            height=0.3\linewidth,
                            angle=0]{lena_noisy_1_m10}
            }
            \hspace{0.0005\linewidth}
        \subfloat[][]
            {
                {\transparent{0.0}
                    \includegraphics[width=0.2828\linewidth,
                                    height=0.3\linewidth,
                                    angle=0]{lena_noisy_1_m10}
                }
            }
    \end{minipage}
    \hfill
    \vspace{0.005\linewidth}
    \begin{minipage}{1.0\textwidth}
        \centering
        \subfloat[][Noise: \Cref{alg:tvs12}]
            {
            \includegraphics[width=0.3\linewidth,
                            height=0.3\linewidth,
                            angle=0]{lena_noisy_1_m3_noise}
            }
            \hspace{0.0005\linewidth}
        \subfloat[][Noise: \Cref{alg:osher}]
            {
            \includegraphics[width=0.3\linewidth,
                            height=0.3\linewidth,
                            angle=0]{lena_noisy_1_m10_noise}
            }
            \hspace{0.0005\linewidth}
        \subfloat[][]
            {
                {\transparent{0.0}
                    \includegraphics[width=0.2828\linewidth,
                                    height=0.3\linewidth,
                                    angle=0]{lena_noisy_1_m10_noise}
                }
            }
    \end{minipage}
    \caption{Comparing the effect of different iterative algorithms on the noisy Lena image. The initial noise level is $7.96$.}
    \label{fig:comparison_n1}
\end{figure}
For convenience in comparing the results from the different algorithms, we collect all the restored images and the clean image together, cf. \cref{fig:comparison_n1}. All the proposed algorithms have a visible improvement in handling the staircase effect compared to \Cref{alg:osher}. Among them, the \Cref{alg:tvs2} results of the best restoration concerning both the PSNR value and visual pleasure.

\begin{figure}[!htbp]
    \captionsetup[subfigure]{
                            justification=centering,
                            labelformat=empty,
                            }
    \centering
    \begin{minipage}{1.0\textwidth}
        \centering
        \subfloat[][Noisy image]
            {
                \begin{tikzpicture}[blankboximg]
                    \node[anchor=south west] (img)
                    {   
                        \includegraphics[width=0.3\linewidth, 
                                        height=0.3\linewidth,
                                        angle=0,
                                        ]{lena_clean}
                    };
                    \draw (img.south west) rectangle (img.north east);
                    \begin{scope}[x=(img.south east),y=(img.north west)]
                        \node[redboximg,draw,minimum height=2.8cm,minimum width=2.8cm] (B1) at (0.5,0.45) {};
                    \end{scope}
                \end{tikzpicture}
            }
            \hspace{0.0005\linewidth}
        \subfloat[][\Cref{alg:osher}]
            {
                \begin{tikzpicture}[blankboximg]
                    \node[anchor=south west] (iimg)
                    {   
                        \includegraphics[width=0.3\linewidth, 
                                        height=0.3\linewidth,
                                        angle=0,
                                        ]{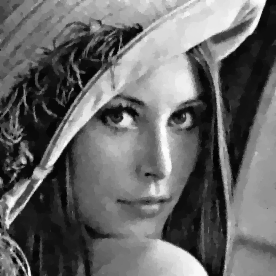}
                    };
                    \draw (iimg.south west) rectangle (iimg.north east);
                    \begin{scope}[x=(iimg.south east),y=(iimg.north west)]
                        \node[blueboximg,draw,minimum height=2.8cm,minimum width=2.8cm] (B2) at (0.5,0.45) {};
                    \end{scope}
                \end{tikzpicture}
            }
            \hspace{0.0005\linewidth}
        \subfloat[][\Cref{alg:tvs2}]
            {
                \begin{tikzpicture}[blankboximg]
                    \node[anchor=south west] (iiimg)
                    {   
                        \includegraphics[width=0.3\linewidth, 
                                        height=0.3\linewidth,
                                        angle=0,
                                        ]{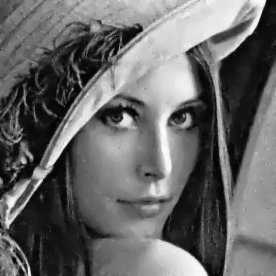}
                    };
                    \draw (iiimg.south west) rectangle (iiimg.north east);
                    \begin{scope}[x=(iiimg.south east),y=(iiimg.north west)]
                        \node[greenboximg,draw,minimum height=2.8cm,minimum width=2.8cm] (B3) at (0.5,0.45) {};
                    \end{scope}
                \end{tikzpicture}
            }
    \end{minipage}
    \hfill
    \vspace{0.005\linewidth}
    \begin{minipage}{1.0\textwidth}
        \centering
        \subfloat[][Part: Clean image]
            {
                \begin{tikzpicture}[redboximg]
                    \node (img1)
                    {   
                        \includegraphics[width=0.3\linewidth, 
                                        height=0.3\linewidth,
                                        angle=0,
                                        clip=true,
                                        trim = 15mm 10mm 15mm 20mm 
                                        ]{lena_clean}
                    };
                    \draw (img1.south west) rectangle (img1.north east);
                \end{tikzpicture}
            }
            \hspace{0.0005\linewidth}
        \subfloat[][Part: \Cref{alg:osher}]
            {
                \begin{tikzpicture}[blueboximg]
                    \node (iimg1)
                    {   
                        \includegraphics[width=0.3\linewidth, 
                                        height=0.3\linewidth,
                                        angle=0,
                                        clip=true,
                                        trim = 15mm 10mm 15mm 20mm 
                                        ]{lena_noisy_2_m10}
                    };
                    \draw (iimg1.south west) rectangle (iimg1.north east);
                \end{tikzpicture}
            }
            \hspace{0.0005\linewidth}
        \subfloat[][Part: \Cref{alg:tvs2}]
            {
                \begin{tikzpicture}[greenboximg]
                    \node (iiimg1)
                    {   
                        \includegraphics[width=0.3\linewidth, 
                                        height=0.3\linewidth,
                                        angle=0,
                                        clip=true,
                                        trim = 15mm 10mm 15mm 20mm 
                                        ]{lena_noisy_2_m2}
                    };
                    \draw (iiimg1.south west) rectangle (iiimg1.north east);
                \end{tikzpicture}
            }
    \end{minipage}
    \caption{Comparing the effect of the proposed Richardson-like iterative \Cref{alg:tvs2} for the second step of the TV-Stokes model and the Osher-like \Cref{alg:osher} for the ROF model on the noisy Lena image. The initial noise level is $15.71$.}
    \label{fig:comparison_n2}
\end{figure}
\begin{figure}[!htbp]
    \captionsetup[subfigure]{
                            justification=centering,
                            labelformat=empty,
                            }
    \centering
    \begin{minipage}{1.0\textwidth}
        \centering
        \subfloat[][Noisy image]
            {
                \begin{tikzpicture}[blankboximg]
                    \node[anchor=south west] (img)
                    {   
                        \includegraphics[width=0.3\linewidth, 
                                        height=0.3\linewidth,
                                        angle=0,
                                        ]{lena_clean}
                    };
                    \draw (img.south west) rectangle (img.north east);
                    \begin{scope}[x=(img.south east),y=(img.north west)]
                        \node[redboximg,draw,minimum height=1.3cm,minimum width=1.3cm] (B1) at (0.5,0.385) {};
                    \end{scope}
                \end{tikzpicture}
            }
            \hspace{0.0005\linewidth}
        \subfloat[][\Cref{alg:osher}]
            {
                    \begin{tikzpicture}[blankboximg]
                        \node[anchor=south west] (iimg)
                        {   
                            \includegraphics[width=0.3\linewidth, 
                                            height=0.3\linewidth,
                                            angle=0,
                                            ]{lena_noisy_2_m10}
                        };
                        \draw (iimg.south west) rectangle (iimg.north east);
                        \begin{scope}[x=(iimg.south east),y=(iimg.north west)]
                            \node[blueboximg,draw,minimum height=1.3cm,minimum width=1.3cm] (B2) at (0.5,0.385) {};
                        \end{scope}
                    \end{tikzpicture}
            }
            \hspace{0.0005\linewidth}
        \subfloat[][\Cref{alg:tvs2}]
            {
                \begin{tikzpicture}[blankboximg]
                    \node[anchor=south west] (iiimg)
                    {   
                        \includegraphics[width=0.3\linewidth, 
                                        height=0.3\linewidth,
                                        angle=0,
                                        ]{lena_noisy_2_m2}
                    };
                    \draw (iiimg.south west) rectangle (iiimg.north east);
                    \begin{scope}[x=(iiimg.south east),y=(iiimg.north west)]
                        \node[selfDefinedColorboximg,draw,minimum height=1.3cm,minimum width=1.3cm] (B3) at (0.5,0.385) {};
                    \end{scope}
                \end{tikzpicture}
            }
    \end{minipage}
    \hfill
    \vspace{0.005\linewidth}
    \begin{minipage}{1.0\textwidth}
        \centering
        \subfloat[][Part: Clean image]
            {
                \begin{tikzpicture}[redboximg]
                    \node (img1)
                    {   
                        \includegraphics[width=0.3\linewidth, 
                                        height=0.3\linewidth,
                                        angle=0,
                                        clip=true,
                                        trim = 15mm 10mm 15mm 20mm 
                                        ]{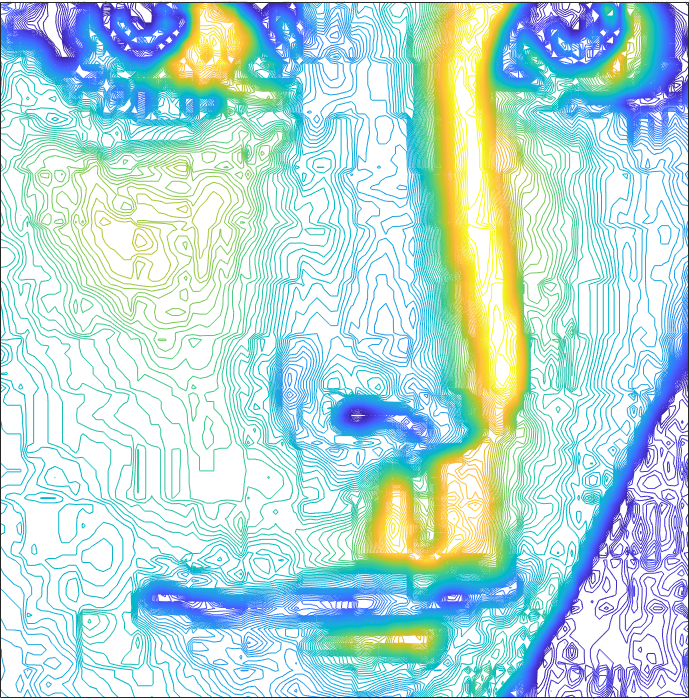}
                    };
                    \draw (img1.south west) rectangle (img1.north east);
                \end{tikzpicture}
            }
            \hspace{0.0005\linewidth}
        \subfloat[][Part: \Cref{alg:osher}]
            {
                \begin{tikzpicture}[blueboximg]
                    \node (iimg1)
                    {   
                        \includegraphics[width=0.3\linewidth, 
                                        height=0.3\linewidth,
                                        angle=0,
                                        clip=true,
                                        trim = 15mm 10mm 15mm 20mm 
                                        ]{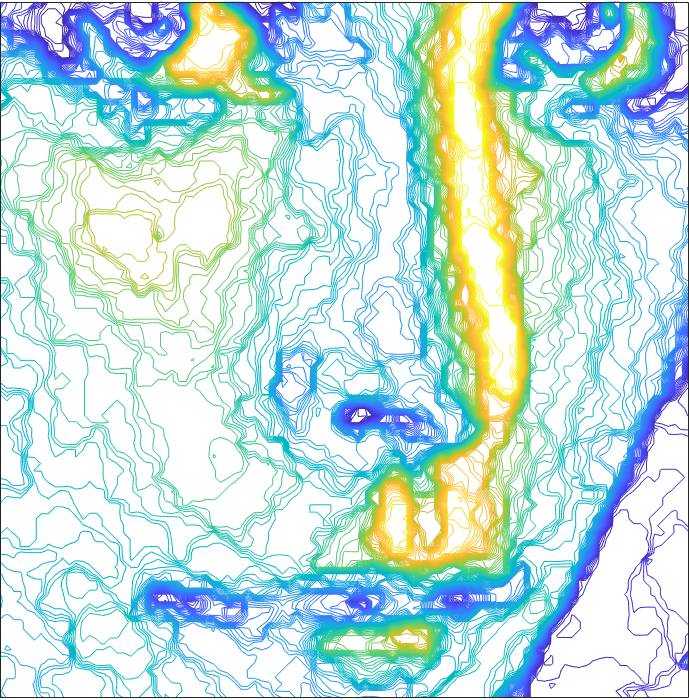}
                    };
                    \draw (iimg1.south west) rectangle (iimg1.north east);
                \end{tikzpicture}
            }
            \hspace{0.0005\linewidth}
        \subfloat[][Part: \Cref{alg:tvs2}]
            {
                \begin{tikzpicture}[selfDefinedColorboximg]
                    \node (iiimg1)
                    {   
                        \includegraphics[width=0.3\linewidth, 
                                        height=0.3\linewidth,
                                        angle=0,
                                        clip=true,
                                        trim = 15mm 10mm 15mm 20mm 
                                        ]{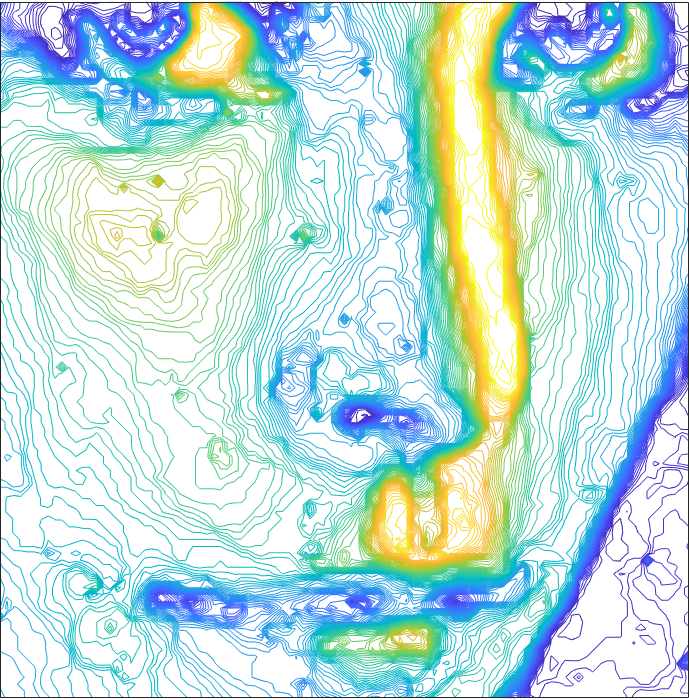}
                    };
                    \draw (iiimg1.south west) rectangle (iiimg1.north east);
                \end{tikzpicture}
            }
    \end{minipage}
    \caption{Contours show, comparing the effect of \Cref{alg:tvs2} and \Cref{alg:osher} on the noisy Lena image. The initial noise level is $15.71$.}
    \label{fig:comparison_n2_contours}
\end{figure}
The other group of experiments is on the same Lena portrait but with heavier noise, cf. \cref{fig:comparison_n2} and \cref{fig:comparison_n2_contours}, where the noise level is at $15.71$ and PSNR is $24.89$. The results shown in \cref{fig:comparison_n2} demonstrate that the proposed \Cref{alg:tvs2} is effective in preserving smooth surfaces like Lena face while the Osher-like iteration, i.e., \Cref{alg:osher}, is defective, showing a patch like surface. The PSNR values of the restored image are $30.23$ and $30.79$ corresponding to \Cref{alg:osher} and \Cref{alg:tvs2}, respectively. The according noise level are $8.50$ and $7.97$. For a better understanding the two restored images, we also plot the contours of these images, cf. \cref{fig:comparison_n2_contours}. The contours obtained from \Cref{alg:tvs2} are of higher parallelity compared to the ones obtained from \Cref{alg:osher}. This reflects also that \Cref{alg:tvs2} can generate smoother surfaces, e.g., Lena face, than \Cref{alg:osher}.

\begin{figure}[!htbp]
    \captionsetup[subfigure]{
                            justification=centering,
                            labelformat=empty,
                            }
    \centering
    \begin{minipage}{1.0\textwidth}
        \centering
        \subfloat[][Clean image]
            {
                \begin{tikzpicture}[blankboximg]
                    \node[anchor=south west] (img)
                    {   
                        \includegraphics[width=0.3\linewidth, 
                                        height=0.3\linewidth,
                                        angle=0,
                                        ]{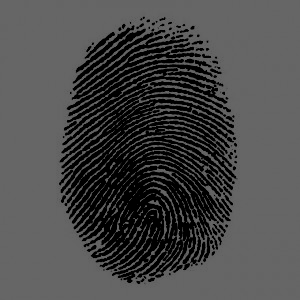}
                    };
                    {\transparent{0.0}
                    \draw (img.south west) rectangle (img.north east);
                    }
                \end{tikzpicture}
            }
            \hspace{0.0005\linewidth}
        \subfloat[][\Cref{alg:osher}]
            {
                \begin{tikzpicture}[blankboximg]
                    \node[anchor=south west] (iimg)
                    {   
                        \includegraphics[width=0.3\linewidth, 
                                        height=0.3\linewidth,
                                        angle=0,
                                        ]{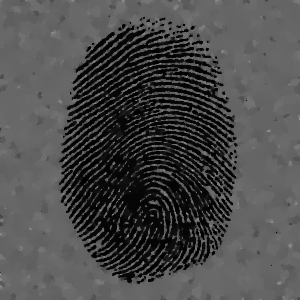}
                    };
                    {\transparent{0.0}
                    \draw (iimg.south west) rectangle (iimg.north east);
                    }
                    \begin{scope}[x=(iimg.south east),y=(iimg.north west)]
                        \node[redboximg,draw,minimum height=2.05cm,minimum width=2.05cm] (B1) at (0.27,0.485) {};
                    \end{scope}
                \end{tikzpicture}
            }
            \hspace{0.0005\linewidth}
        \subfloat[][\Cref{alg:tvs2}]
            {
                \begin{tikzpicture}[blankboximg]
                    \node[anchor=south west] (iiimg)
                    {   
                        \includegraphics[width=0.3\linewidth, 
                                        height=0.3\linewidth,
                                        angle=0,
                                        ]{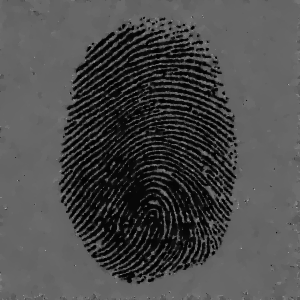}
                    };
                    {\transparent{0.0}
                    \draw (iiimg.south west) rectangle (iiimg.north east);
                    }
                    \begin{scope}[x=(iiimg.south east),y=(iiimg.north west)]
                        \node[blueboximg,draw,minimum height=2.05cm,minimum width=2.05cm] (B2) at (0.27,0.485) {};
                    \end{scope}
                \end{tikzpicture}
            }
    \end{minipage}
    \hfill
    \vspace{0.005\linewidth}
    \begin{minipage}{1.0\textwidth}
        \centering
        \subfloat[][Initial Gaussian noise]
            {
                \begin{tikzpicture}[blankboximg]
                    \node[anchor=south west] (img2)
                    {   
                        \includegraphics[width=0.3\linewidth, 
                                        height=0.3\linewidth,
                                        angle=0,
                                        ]{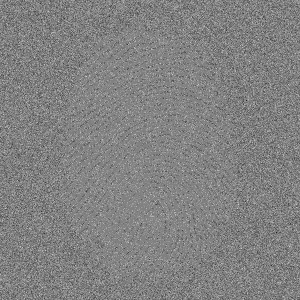}
                    };
                    {\transparent{0.0}
                    \draw (img2.south west) rectangle (img2.north east);
                    }
                \end{tikzpicture}
            }
            \hspace{0.0005\linewidth}
        \subfloat[][Noise: \Cref{alg:osher}]
            {
                \begin{tikzpicture}[blankboximg]
                    \node[anchor=south west] (iimg2)
                    {   
                        \includegraphics[width=0.3\linewidth, 
                                        height=0.3\linewidth,
                                        angle=0,
                                        ]{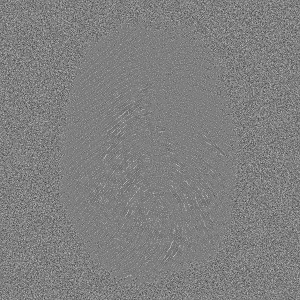}
                    };
                    {\transparent{0.0}
                    \draw (iimg2.south west) rectangle (iimg2.north east);
                    }
                \end{tikzpicture}
            }
            \hspace{0.0005\linewidth}
        \subfloat[][Noise: \Cref{alg:tvs2}]
            {
                \begin{tikzpicture}[blankboximg]
                    \node[anchor=south west] (iiimg2)
                    {   
                        \includegraphics[width=0.3\linewidth, 
                                        height=0.3\linewidth,
                                        angle=0,
                                        ]{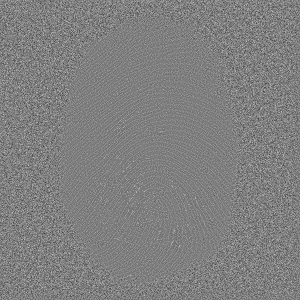}
                    };
                    {\transparent{0.0}
                    \draw (iiimg2.south west) rectangle (iiimg2.north east);
                    }
                \end{tikzpicture}
            }
    \end{minipage}
    \hfill
    \vspace{0.005\linewidth}
    \begin{minipage}{1.0\textwidth}
        \centering
        \subfloat[][Noisy image]
            {
                \begin{tikzpicture}[blankboximg]
                    \node[anchor=south west] (img1)
                    {   
                        \includegraphics[width=0.3\linewidth, 
                                        height=0.3\linewidth,
                                        angle=0,
                                        ]{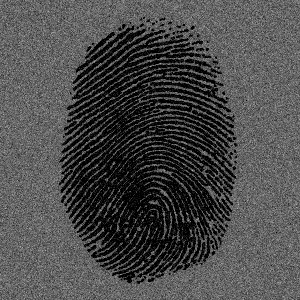}
                    };
                    {\transparent{0.0}
                    \draw (img1.south west) rectangle (img1.north east);
                    }
                \end{tikzpicture}
            }
            \hspace{0.0005\linewidth}
        \subfloat[][]
            {
                \begin{tikzpicture}[redboximg]
                    \node[anchor=south west] (iimg1)
                    {   
                        \includegraphics[width=0.3\linewidth, 
                                        height=0.3\linewidth,
                                        angle=0,
                                        clip=true,
                                        trim = 0mm 25mm 50mm 25mm 
                                        ]{fingerprint_noisy_1_m10}
                    };
                    \draw (iimg1.south west) rectangle (iimg1.north east);
                \end{tikzpicture}
            }
            \hspace{0.0005\linewidth}
        \subfloat[][]
            {
                \begin{tikzpicture}[blueboximg]
                    \node[anchor=south west] (iiimg1)
                    {   
                        \includegraphics[width=0.3\linewidth, 
                                        height=0.3\linewidth,
                                        angle=0,
                                        clip=true,
                                        trim = 0mm 25mm 50mm 25mm 
                                        ]{fingerprint_noisy_1_m2}
                    };
                    \draw (iiimg1.south west) rectangle (iiimg1.north east);
                \end{tikzpicture}
            }
    \end{minipage}
    \caption{Comparing the effect of iterative \Cref{alg:tvs2} and Osher-like \Cref{alg:osher} for ROF on the noisy fingerprint image. The initial noise level is $22.11$.}
    \label{fig:comparison_fingerprint}
\end{figure}
\begin{figure}[!htbp]
    \captionsetup[subfigure]{
                            justification=centering,
                            labelformat=empty,
                            }
    \centering
    \begin{minipage}{1.0\textwidth}
        \centering
        \subfloat[][Part 1: Fingerprint: \Cref{alg:osher}]
            {
                \begin{tikzpicture}[blankboximg]
                    \node[anchor=south west] (img)
                    {   
                        \includegraphics[width=0.3\linewidth, 
                                        height=0.3\linewidth,
                                        angle=0,
                                        ]{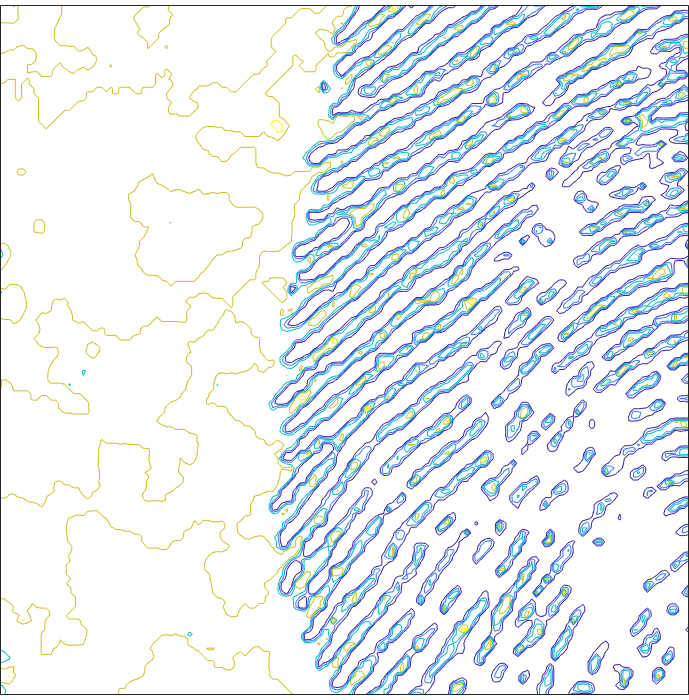}
                    };
                    {\transparent{0.0}
                    \draw (img.south west) rectangle (img.north east);
                    }
                    \begin{scope}[x=(img.south east),y=(img.north west)]
                        \node[redboximg,draw,minimum height=1.95cm,minimum width=1.95cm] (B1) at (0.745,0.7375) {};
                    \end{scope}
                \end{tikzpicture}
            }
            \hspace{0.0005\linewidth}
        \subfloat[][Part 1: Fingerprint: \Cref{alg:tvs2}]
            {
                \begin{tikzpicture}[blankboximg]
                    \node[anchor=south west] (iimg)
                    {   
                        \includegraphics[width=0.3\linewidth, 
                                        height=0.3\linewidth,
                                        angle=0,
                                        ]{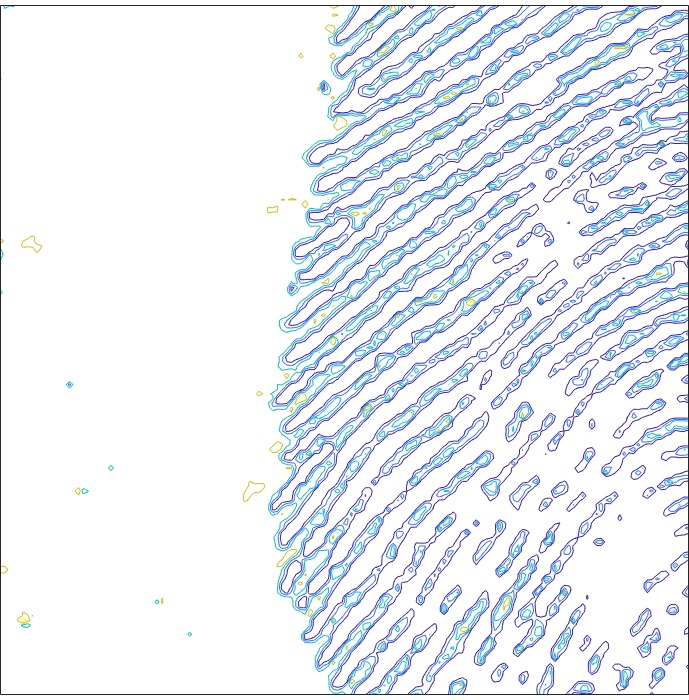}
                    };
                    {\transparent{0.0}
                    \draw (iimg.south west) rectangle (iimg.north east);
                    }
                    \begin{scope}[x=(iimg.south east),y=(iimg.north west)]
                        \node[blueboximg,draw,minimum height=1.95cm,minimum width=1.95cm] (B2) at (0.745,0.7375) {};
                    \end{scope}
                \end{tikzpicture}
            }
            \hspace{0.0005\linewidth}
        \subfloat[][Part 1: Fingerprint: Clean]
            {
                \begin{tikzpicture}[blankboximg]
                    \node[anchor=south west] (iiimg)
                    {   
                        \includegraphics[width=0.3\linewidth, 
                                        height=0.3\linewidth,
                                        angle=0,
                                        ]{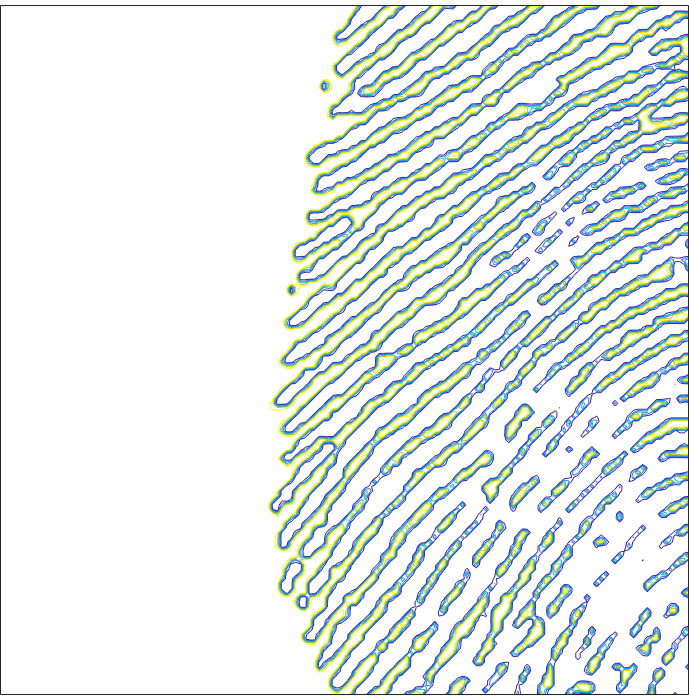}
                    };
                    {\transparent{0.0}
                    \draw (iiimg.south west) rectangle (iiimg.north east);
                    }
                    \begin{scope}[x=(iiimg.south east),y=(iiimg.north west)]
                        \node[selfDefinedColorboximg,draw,minimum height=1.95cm,minimum width=1.95cm] (B3) at (0.745,0.7375) {};
                    \end{scope}
                \end{tikzpicture}
            }
    \end{minipage}
    \hfill
    \vspace{0.005\linewidth}
    \begin{minipage}{1.0\textwidth}
        \centering
        \subfloat[][Part 2: Fingerprint: \Cref{alg:osher}]
            {
                \begin{tikzpicture}[redboximg]
                    \node[anchor=south west] (img1)
                    {   
                        \includegraphics[width=0.3\linewidth, 
                                        height=0.3\linewidth,
                                        angle=0,
                                        ]{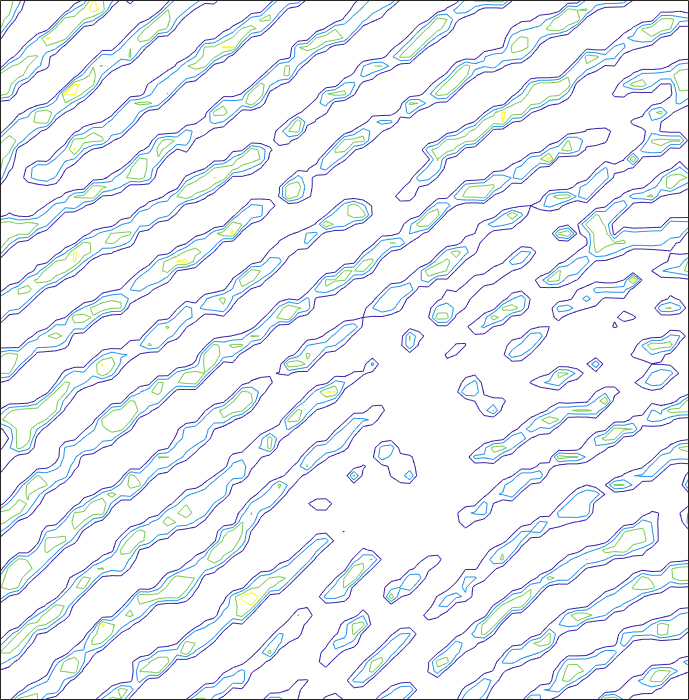}
                    };
                    \draw (img1.south west) rectangle (img1.north east);
                \end{tikzpicture}
                \label{subfig:}
            }
            \hspace{0.0005\linewidth}
        \subfloat[][Part 2: Fingerprint: \Cref{alg:tvs2}]
            {
                \begin{tikzpicture}[blueboximg]
                    \node[anchor=south west] (iimg1)
                    {   
                        \includegraphics[width=0.3\linewidth, 
                                        height=0.3\linewidth,
                                        angle=0,
                                        ]{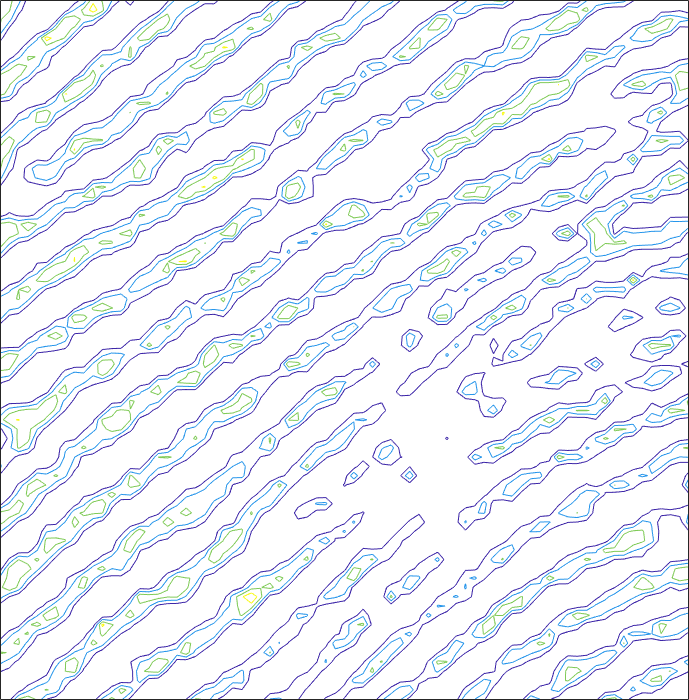}
                    };
                    \draw (iimg1.south west) rectangle (iimg1.north east);
                \end{tikzpicture}
            }
            \hspace{0.0005\linewidth}
        \subfloat[][Part 2: Fingerprint: Clean]
            {
                \begin{tikzpicture}[selfDefinedColorboximg]
                    \node[anchor=south west] (iiimg1)
                    {   
                        \includegraphics[width=0.3\linewidth, 
                                        height=0.3\linewidth,
                                        angle=0,
                                        ]{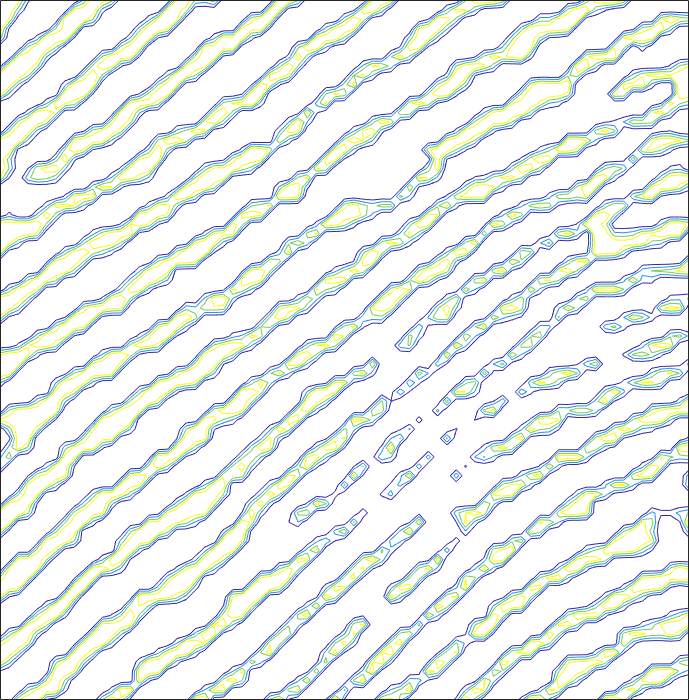}
                    };
                    \draw (iiimg1.south west) rectangle (iiimg1.north east);
                \end{tikzpicture}
            }
    \end{minipage}
    \caption{Contours show, comparing the effect of iterative \Cref{alg:tvs2} and Osher-like \Cref{alg:osher} for the ROF model on the noisy fingerprint image. The initial noise level is $22.11$.}
    \label{fig:comparison_contour_2}
\end{figure}
The next experiments are taken on a image of a fingerprint, cf. \cref{fig:comparison_fingerprint} and \cref{fig:comparison_contour_2}, with PSNR $22.65$ at noise level $22.11$. The results demonstrate the effectiveness of the proposed \Cref{alg:tvs2} in preserving sharp edges such as fingerprint textures and in restoring the smoothed structures like the surroundings in this image. The contours show a better connectivity of the texture for \cref{fig:comparison_contour_2} resulting longer structures, cf. \cref{fig:comparison_contour_2}. 

\begin{figure}[!htbp]
\captionsetup[subfigure]{
                        justification=centering,
                        labelformat=empty,
                        }
\centering
\begin{minipage}{1.0\textwidth}
    \centering
    \subfloat[][\Cref{alg:osher}]
        {
            \begin{tikzpicture}[blankboximg]
                \node[anchor=south west] (img)
                {   
                    \includegraphics[width=0.3\linewidth, 
                                    height=0.3\linewidth,
                                    angle=0,
                                    ]{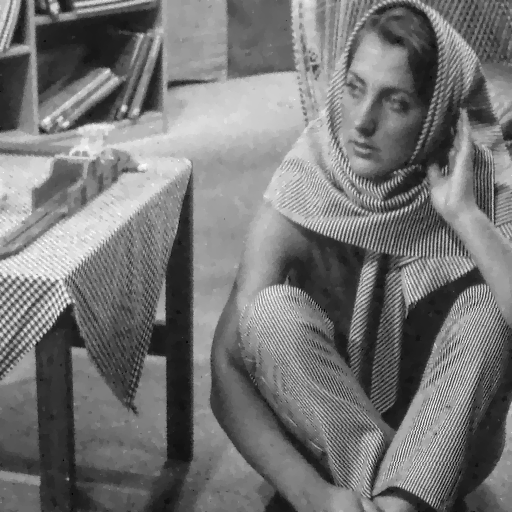}
                };
                \draw (img.south west) rectangle (img.north east);
                \begin{scope}[x=(img.south east),y=(img.north west)]
                    \node[redboximg,draw,minimum height=1.45cm,minimum width=1.45cm] (B1) at (0.78,0.815) {};
                    \node[blueboximg,draw,minimum height=1.55cm,minimum width=1.55cm] (B2) at (0.51,0.42) {};
                \end{scope}
            \end{tikzpicture}
        }
        \hspace{0.0005\linewidth}
    \subfloat[][]
        {
            \begin{tikzpicture}[redboximg]
                \node (img1)
                {   
                    \includegraphics[width=0.3\linewidth, 
                                    height=0.3\linewidth,
                                    angle=0,
                                    clip=true,
                                    trim = 110mm 115mm 5mm 0mm 
                                    ]{barbara_noisy_1_m10}
                };
                \draw (img1.south west) rectangle (img1.north east);
            \end{tikzpicture}
        }
        \hspace{0.0005\linewidth}
    \subfloat[][]
        {
            \begin{tikzpicture}[selfDefinedColorboximg]
                \node (iimg1)
                {   
                    \includegraphics[width=0.3\linewidth, 
                                    height=0.3\linewidth,
                                    angle=0,
                                    clip=true,
                                    trim = 110mm 115mm 5mm 0mm 
                                    ]{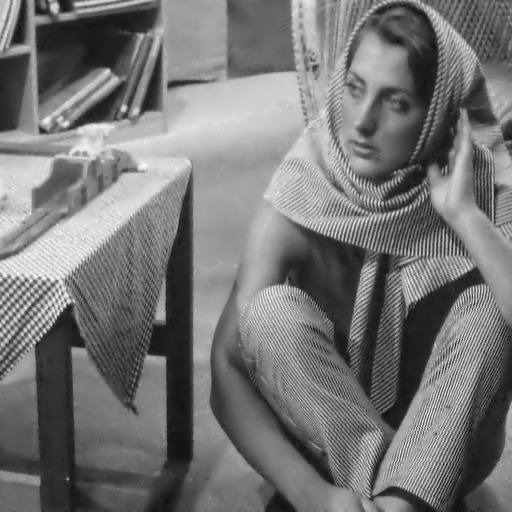}
                };
                \draw (iimg1.south west) rectangle (iimg1.north east);
            \end{tikzpicture}
        }
\end{minipage}
    \hfill
    \vspace{0.005\linewidth}
\begin{minipage}{1.0\textwidth}
    \centering
    \subfloat[][]
        {
            \begin{tikzpicture}[blueboximg]
                \node(img2)
                {   
                    \includegraphics[width=0.3\linewidth, 
                                    height=0.3\linewidth,
                                    angle=0,
                                    clip=true,
                                    trim = 57mm 43mm 54mm 68mm 
                                    ]{barbara_noisy_1_m10}
                };
                \draw (img2.south west) rectangle (img2.north east);
            \end{tikzpicture}
        }
        \hspace{0.0005\linewidth}
    \subfloat[][]
        {
            \begin{tikzpicture}[greenboximg]
                \node(iimg2)
                {   
                    \includegraphics[width=0.3\linewidth, 
                                    height=0.3\linewidth,
                                    angle=0,
                                    clip=true,
                                    trim = 57mm 43mm 54mm 68mm 
                                    ]{barbara_noisy_1_m2}
                };
                \draw (iimg2.south west) rectangle (iimg2.north east);
            \end{tikzpicture}
        }
        \hspace{0.0005\linewidth}
    \subfloat[][\Cref{alg:tvs2}]
        {
            \begin{tikzpicture}[blankboximg]
                \node[anchor=south west] (iimg)
                {   
                    \includegraphics[width=0.3\linewidth, 
                                    height=0.3\linewidth,
                                    angle=0,
                                    ]{barbara_noisy_1_m2}
                };
                \draw (iimg.south west) rectangle (iimg.north east);
                \begin{scope}[x=(iimg.south east),y=(iimg.north west)]
                    \node[selfDefinedColorboximg,draw,minimum height=1.45cm,minimum width=1.45cm] (B11) at (0.78,0.815) {};
                    \node[greenboximg,draw,minimum height=1.55cm,minimum width=1.55cm] (B12) at (0.51,0.42) {};
                \end{scope}
            \end{tikzpicture}
        }
\end{minipage}
\caption{Showing the effect of using \Cref{alg:tvs2} and \Cref{alg:osher} on Barbara image.}
\label{fig:com_barbara}
\end{figure}
\begin{figure}[!htbp]
    \captionsetup[subfigure]{
                            justification=centering,
                            labelformat=empty,
                            }
    \centering
    \begin{minipage}{1.0\textwidth}
        \centering
        \subfloat[][\Cref{alg:osher}]
            {
                \begin{tikzpicture}[blankboximg]
                    \node[anchor=south west] (img)
                    {   
                        \includegraphics[width=0.3\linewidth, 
                                        height=0.3\linewidth,
                                        angle=0,
                                        ]{barbara_noisy_1_m10}
                    };
                    \draw (img.south west) rectangle (img.north east);
                    \begin{scope}[x=(img.south east),y=(img.north west)]
                    \node[redboximg,draw,minimum height=1.08cm,minimum width=1.08cm] (B1) at (0.78,0.775) {};
                    \node[blueboximg,draw,minimum height=1.55cm,minimum width=1.55cm] (B2) at (0.51,0.42) {};
                    \end{scope}
                \end{tikzpicture}
            }
            \hspace{0.0005\linewidth}
        \subfloat[][]
            {
                \begin{tikzpicture}[redboximg]
                    \node (img1)
                    {   
                        \includegraphics[width=0.3\linewidth, 
                                        height=0.3\linewidth,
                                        angle=0,
                                        ]{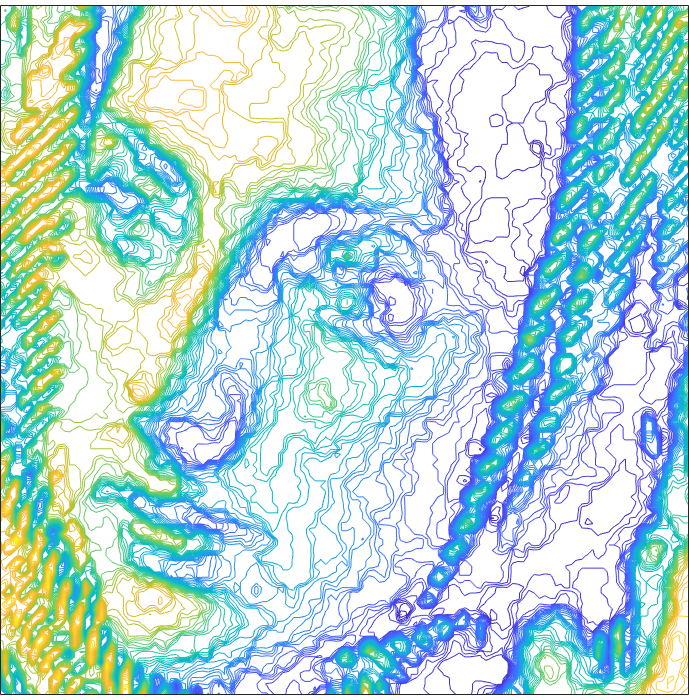}
                    };
                    \draw (img1.south west) rectangle (img1.north east);
                \end{tikzpicture}
            }
            \hspace{0.0005\linewidth}
        \subfloat[][]
            {
                \begin{tikzpicture}[selfDefinedColorboximg]
                    \node (iimg1)
                    {   
                        \includegraphics[width=0.3\linewidth, 
                                        height=0.3\linewidth,
                                        angle=0,
                                        ]{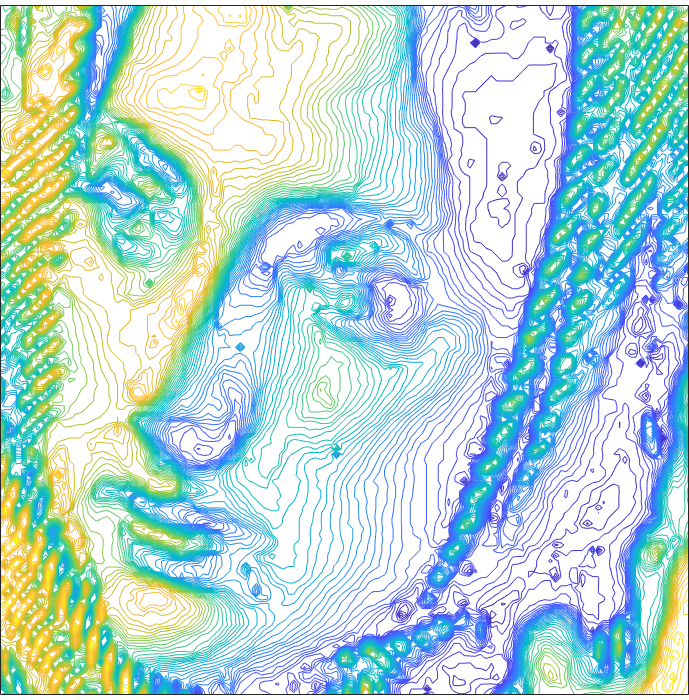}
                    };
                    \draw (iimg1.south west) rectangle (iimg1.north east);
                \end{tikzpicture}
            }
    \end{minipage}
    \hfill
    \vspace{0.005\linewidth}
    \begin{minipage}{1.0\textwidth}
        \centering
        \subfloat[][]
            {
                \begin{tikzpicture}[blueboximg]
                    \node(img2)
                    {   
                        \includegraphics[width=0.3\linewidth, 
                                        height=0.3\linewidth,
                                        angle=0,
                                        ]{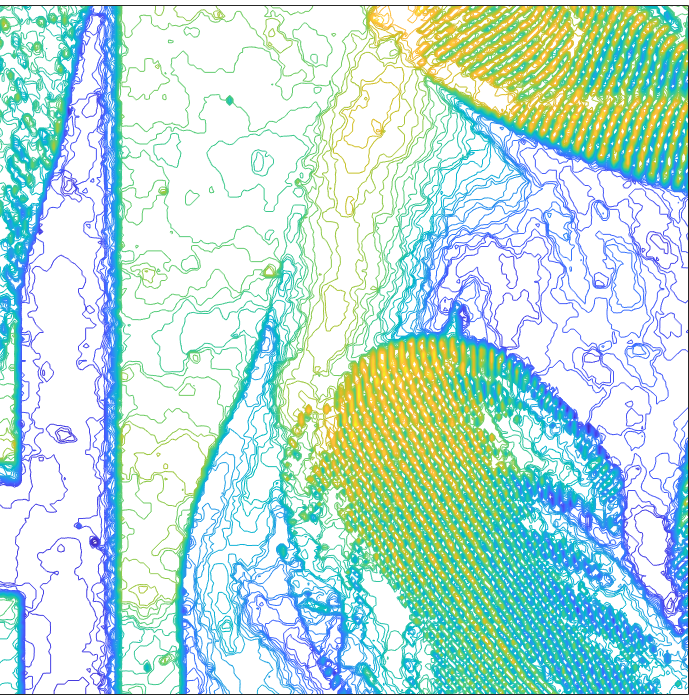}
                    };
                    \draw (img2.south west) rectangle (img2.north east);
                \end{tikzpicture}
            }
            \hspace{0.0005\linewidth}
        \subfloat[][]
            {
                \begin{tikzpicture}[greenboximg]
                    \node(iimg2)
                    {   
                        \includegraphics[width=0.3\linewidth, 
                                        height=0.3\linewidth,
                                        angle=0,
                                        ]{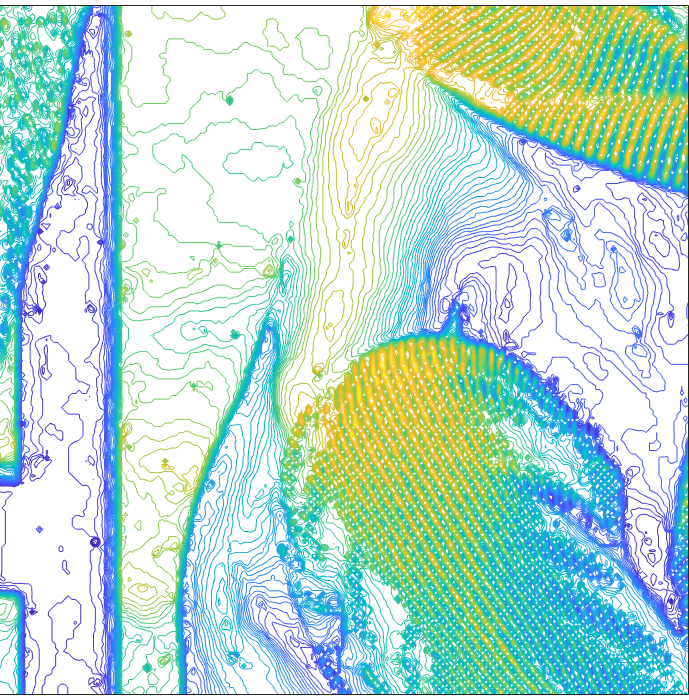}
                    };
                    \draw (iimg2.south west) rectangle (iimg2.north east);
                \end{tikzpicture}
            }
            \hspace{0.0005\linewidth}
        \subfloat[][\Cref{alg:tvs2}]
            {
                \begin{tikzpicture}[blankboximg]
                    \node[anchor=south west] (iimg)
                    {   
                        \includegraphics[width=0.3\linewidth, 
                                        height=0.3\linewidth,
                                        angle=0,
                                        ]{barbara_noisy_1_m2}
                    };
                    \draw (iimg.south west) rectangle (iimg.north east);
                    \begin{scope}[x=(iimg.south east),y=(iimg.north west)]
                    \node[selfDefinedColorboximg,draw,minimum height=1.08cm,minimum width=1.08cm] (B11) at (0.78,0.775) {};
                    \node[greenboximg,draw,minimum height=1.55cm,minimum width=1.55cm] (B12) at (0.51,0.42) {};
                    \end{scope}
                \end{tikzpicture}
            }
    \end{minipage}
    \caption{The contours showing the effect of using \Cref{alg:tvs2} and \Cref{alg:osher} on Barbara image.}
    \label{fig:com_barbara_contour}
\end{figure}
The last experiments are applied on the Barbara image, cf. \cref{fig:com_barbara} and \cref{fig:com_barbara_contour}. The noise level of the initial image is $20.06$ and initial PSNR is $28.14$. The results from the proposed \Cref{alg:tvs2} show much smoother face and arm compared to the results from \Cref{alg:osher}. The results of \Cref{alg:tvs2} also show a better restoration in preserving pinstripes structures. The final restored images are at noise level $14.15$ and $12.91$ for \Cref{alg:osher} and \Cref{alg:tvs2}, respectively. The according PSNR are $31.17$ and $31.97$.

\section{Conclusions}
\label{sec:conclusions}
In this paper, we have proposed the Richardson-like iterative algorithms applied to the first step, the second step, and both steps of the TV-Stokes model, respectively. We have proven the well-definedness and the convergence of each algorithm. The numerical experiments show a visible improvement compared to the Osher-like iteration on the ROF model in both edge preserving and surface smoothing.

\bibliographystyle{siamplain}
\bibliography{refs}

\end{document}